\documentclass{article}
\usepackage{amsmath,amssymb,amsthm,mathrsfs, adjustbox, enumerate, amsbsy, stmaryrd, stackrel, nccmath,
color,times,textcomp,verbatim,yfonts}
\usepackage{mathtools}% for "\span"
\usepackage{amsfonts}
\usepackage{array}

\usepackage{arydshln}%dashed lines in matrix

\usepackage{authblk}

\usepackage{euscript}%for \EuScript{C}

\usepackage{nicematrix}
\NiceMatrixOptions{nullify-dots}%for matrix Ddots

\usepackage[utf8]{inputenc}

\usepackage{wasysym}
\usepackage{comment}
\usepackage[]{enumitem} % custom listing
\usepackage{caption}

\usepackage{graphicx}
\usepackage{amscd}
\usepackage{epic}
\usepackage[all, color,matrix,arrow]{xy}
\usepackage[all,color]{xy}
\usepackage{color}
\usepackage{cancel}
\usepackage{hyperref}

\hypersetup{
    colorlinks=true, %set true if you want colored links
    linkcolor=blue,  %choose some color if you want links to stand out
   citecolor=blue,
}

\usepackage{graphpap, color}
\usepackage{mathrsfs}%,mathabx
\usepackage{graphicx}
\usepackage{tikz,tikz-cd}
\usepackage{imakeidx}

\usepackage[toc,title,page]{appendix}

\usepackage[nottoc]{tocbibind}

\usepackage{titlesec}

\usepackage{subcaption,graphicx}

\newcommand{\rulesep}{\unskip\ \vrule\ }

\newtheorem{prop}{Proposition}[section]

\newtheorem{conj}[prop]{Conjecture}

\newtheorem{coro}[prop]{Corollary}

\newtheorem{lemm}[prop]{Lemma}

\newtheorem{pf-thm}[prop]{proof of theorem}

\newtheorem{theo}[prop]{Theorem}
\newtheorem{prob}[prop]{Problem}

\newtheorem{mainlemm}[prop]{Main Lemma}
\newtheorem*{ack}{Acknowledgments}

\theoremstyle{definition}

\newtheorem{defi}[prop]{Definition}
\newtheorem{exam}[prop]{Example}
\newtheorem{rema}[prop]{Remark}

\allowdisplaybreaks[4]

\def\RR{\mathbb{R}}

\def\bL{\mathbf{L}}

\def\bx{\mathbf{x}}

\def\diag{\mathrm{diag}}

\def\bP{\boldsymbol{\mathrm{P}}}
\def\bL{\boldsymbol{\mathrm{L}}}
\def\bY{\boldsymbol{\mathrm{Y}}}
\def\bZ{\boldsymbol{\mathrm{Z}}}
\def\PRs{\boldsymbol{\mathrm{P}}(\mathbb{R}_{\mathrm{sym}}^{n\times n})}

\def\RPn{\mathbb{R}\boldsymbol{\mathrm{P}}^n}
\def\RP{\mathbb{R}\boldsymbol{\mathrm{P}}}
\def\Rsn{\mathbb{R}_{\mathrm{sym}}^{n\times n}}
\def\EC{\EuScript{C}}

\def\sym{\mathrm{Sym}}

\def\Om{\Omega}

\DeclarePairedDelimiterX{\norm}[1]{\lVert}{\rVert}{#1}

\makeindex

\numberwithin{equation}{section}

\title{A Decomposition Lemma in Convex Integration via Classical Algebraic Geometry}

\date{}

\author{ Zhitong Su, Weijun Zhang}

\newcommand{\Addresses}{{% additional braces for segregating \footnotesize

  \bigskip
  \footnotesize
  \begin{flushright}Zhitong Su\\ 
  \textsc{MOE-LCSM\\ School of Mathematics and Statistics\\ Hunan Normal University\\ Changsha 410081, P. R. China}\par\nopagebreak
  \textit{E-mail address}:  \texttt{suzht@hunnu.edu.cn}\end{flushright}

 \bigskip
  \footnotesize
  \begin{flushright}Weijun Zhang\\ 
  \textsc{School of Mathematics and Statistics\\ Nanfang College, Guangzhou\\ Guangzhou 510970, P.R. China}\par\nopagebreak
  \textit{E-mail address}:  \texttt{zwj2017@amss.ac.cn}\end{flushright}

}}

\begin{document}

\maketitle
\begin{abstract}

In this paper, we prove a decomposition lemma for symmetric matrix fields on bounded domains:
$ D+\sym\nabla\Phi=\sum_i a_i^2\xi_i\otimes\xi_i$
with uniform control on $\Phi$ and $a_i^2$, using  fewer than the usual $n(n+1)/2$ rank-one symmetric terms. Except possibly in dimensions $n=8,16$, the decomposition is shown to be
optimal through algebraic arguments. This reduces the number of steps in convex integration for a nonlinear PDE system, improving H\"older regularity of flexible solutions in dimension $n\ge3$. This PDE is a partial linearization of the codimension-one local isometric embedding equation in the Nash–Kuiper theorem, and also yields improved regularity for very weak solutions of related $2$D Monge--Amp\'ere and $2$-Hessian systems. The improved H\"older exponent is any $\alpha<(n^2+1)^{-1}$ for $n=2,4,8,16$ and any $\alpha<(n^2+n-2\rho(n/2)-1)^{-1}$ otherwise, where $\rho$ is the Radon--Hurwitz number,  related to Bott periodicity.

The proof involves novel applications of  algebraic geometry and topology that yield the optimality of decomposition,  including Adams' theorem on vector fields on spheres,  intersections of projective
varieties, and projective duality, combined with an elliptic method that avoids loss of differentiability.

\end{abstract}

\setcounter{section}{0}

\section{Introduction}
The interplay between flexibility and rigidity is a theme across multiple disciplines in modern mathematics. The  phenomenon of flexibility in analytic problems was first discovered by Nash in his celebrated work \cite{Nash_C1} on 
$C^1$ isometric embeddings. To provide a better understanding of such phenomena and establish a general framework for solving flexible analytic problems, Gromov introduced the \textit{h-principle} (see \cite{gromov1986partial,Cieliebak_Eliashberg_Mishachev2024introduction}), and reformulated Nash's idea into the method of \textit{convex integration}, which is applicable to a broader class of problems.\par

Motivated by this flexible behavior and the convex integration framework, this paper focuses on the following nonlinear equations on a bounded domain $\Omega\subset\mathbb{R}^n$ with smooth boundary:

\begin{equation}\label{ddagger}
    \begin{aligned}
        &\text{Find}\quad  v:\Omega\longrightarrow \mathbb{R}, \quad w: \Omega \longrightarrow \mathbb{R}^n\quad \text{satisfying}\\
&\quad \frac{1}{2} \nabla v\otimes \nabla v+\sym\nabla w=A,
    \end{aligned}\tag{$\ddagger$}
\end{equation}
with given $A:\Omega\longrightarrow \Rsn$.\par 
Equation \eqref{ddagger} has a close relation with the codimension-one local isometric embedding equation in the famous Nash--Kuiper \cite{Nash_C1,Kuiper1955OnC1_I&II} theorem:
\begin{equation}\label{NK}
\begin{aligned}
   \text{Fi}&\text{nd }u:\Omega\longrightarrow \mathbb{R}^{n+1} \quad \text{ satisfying}\\
   &\quad \nabla u\otimes\nabla u=g \quad\text{on } \Omega 
\end{aligned}
    \tag{NK}
\end{equation}
with $g$ being a Riemannian metric of $\Omega$ viewed as a $(0,2)$ type tensor. In fact, consider a perturbation  $(g_{ij})(t)=I_n+2t^2A+o(t^2)$ as $t\rightarrow 0$, and let $u=[x_1+t^2w_1,\cdots,x_n+t^2w_n,tv]$. Then, the $t^2$ terms in \eqref{NK} reduce to equation \eqref{ddagger}; in this sense, the equation \eqref{ddagger} can be viewed as a partial linearization of local isometric embedding equation \eqref{NK}. See \cite{Daniel_Hornung, lewicka2025monge} for related discussions. \par

One classical example on the rigidity side of isometric embedding is a rigidity result due to Cohn-Vossen and Herglotz  on Weyl's problem:   for $(S^2,g)$ with positive Gaussian curvature, its isometric  embedding into $\mathbb{R}^3$ is unique up to rigid motions.
 In the 1950s, Borisov extended the rigidity results to the embeddings of $C^{1,\alpha}$ for  $\alpha>\frac{2}{3}$ (see  \cite[Section 2.3]{han_isometric_2023}). On the flexibility side, in 2012, Conti, De Lellis, and Sz\'ekelyhidi showed in \cite{conti_h-principle_2012} that for any $\alpha<\frac{1}{n^2+n+1}$, the equation \eqref{NK} admits $C^{1,\alpha}$ flexible solutions in the sense that, for any $u^\flat\in C^0(\overline{\Omega},\mathbb{R}^{n+1})$ and $\epsilon>0$, there exists a solution $u\in C^{1,\alpha}(\overline{\Omega},\mathbb{R}^{n+1})$ to \eqref{NK} such that $\|u-u^\flat\|_0<\epsilon$, as in the Nash--Kuiper theorem. 
In \cite{DeLellis2015ANT}, the H\"older exponent was further improved to $\alpha<\frac{1}{5}$ in the case $n=2$. For detailed discussions, we refer to the monograph of Han and Hong
\cite{hanHong2006isometric} and the article of Han and Lin
\cite{hanlin2008isometric}. Very recently, in \cite{cao2025nashkuipertheoremisometricimmersions}, Cao, Hirsch, and Inauen made a breakthrough to push the highest exponent to $\alpha<\frac{1}{n^2-n+1}$. The above results are obtained thanks to  the modified convex integration method  originating in Nash's work. \par
The precise threshold between flexibility and rigidity in terms of $\alpha$ remains unknown; it was posed as Question 27 in Yau's Open Problems in Geometry \cite{YAU1993OPEN}. Convex integration for $C^{1,\alpha}$ isometric immersions parallels the scheme used in Onsager’s conjecture \cite{DeLellisICCM}, suggesting $\alpha<\frac{1}{3}$ as a possible threshold for flexibility \cite[Section 1]{cao2025nashkuipertheoremisometricimmersions}. By contrast, Gromov \cite[Question 39]{Gromov_on_nash}, as well as De Lellis and Sz\'ekelyhidi \cite[Section 10.1]{DeLellis_Szekelyhidi_on_nash}, conjectured the critical value to be $\alpha=\frac{1}{2}$.\par

Motivated by obtaining  flexible solutions with higher-regularity to \eqref{ddagger} and \eqref{NK}, we use an elliptic system to reduce the number of rank-one matrices in the decomposition, thereby decreasing the steps in each stage of convex integration and improving regularity.

Before stating the main lemma, some necessary notions are introduced.  We denote by $\mathbb{R}^{n\times n}$  the space of all $n\times n$ matrices, and  $\Rsn$  the space of all symmetric ones, and denote by $\sym (A):=\frac{1}{2}(A+A^T)$ for any $A\in \mathbb{R}^{n\times n}$. Recall that any rank-one symmetric matrix is of the form $\xi\otimes\xi$ for $\xi\in\mathbb{R}^n$. More importantly, we introduce the following integer $\Xi_n$, which represents the number of
rank-one terms in the decomposition in our paper:
\begin{equation}\label{defi of Xi}
\Xi_n:=\begin{cases}
\frac{n(n+1)}{2}-\rho(\frac{1}{2}n)    & \text{for }  n=2, 4,8,16,\\
\frac{n(n+1)}{2}-\rho(\frac{1}{2}n)-1 & \text{for other }n\in \mathbb{Z}_{\geq 2},
    \end{cases}    
\end{equation}
where $\rho(\frac{1}{2}n)$ is the Radon--Hurwitz number, defined as follows. 
\begin{defi}\label{R-H number defi}
    For $n\in\mathbb{Z}_{\geq1}$, write $n=2^{4a+b}(2c+1)$ with $a,b,c\in \mathbb{Z}_{\geq 0}$ and $3\geq b\geq 0$. Then \textbf{Radon--Hurwitz number $\rho(n)$} is defined by
    \begin{equation}
        \rho(n):=8a+2^b.
    \end{equation}
    When $\rho(\frac{1}{2} n)$ appears and $n$ is odd, we use the convention $\rho(\frac{1}{2} n)=0$ throughout.
\end{defi}
  
\begin{table}[h!]
    \centering
    \begin{tabular}{ c|c c c c c c c c c c c} 
 $n$&$ 1 $& $2$ & $3$ & $4$& $5$& $6$& $7$& $8$& $\cdots$&$16$& $\cdots$ \\ 
 \hline
$\rho(n)$&$ 1 $& $2$ & $1$ & $4$& $1$& $2$& $1$& $8$& $\cdots$&$9$& $\cdots$ \\ 

 $\frac{n(n+1)}{2}$&  & $3$ & $6$ & $10$& $15$& $21$& $28$& $36$& $\cdots$&$136$& $\cdots$ \\
\hline
$\Xi_n$&$  $& $2$ & $5$ & $8 $& $14$& $19$& $27$& $32$& $\cdots$&$128$& $\cdots$
\end{tabular}
    \caption{First few values of $\rho(n)$ and $\Xi_n$.}
   
\end{table}

It is noteworthy that $\rho(n)$ encodes an $8$-fold periodicity as $$\rho(16n)=\rho(n)+8.$$

The main ingredient of this paper is the following decomposition lemma, which underlies the convex integration scheme.
\begin{mainlemm}[Decomposition Lemma]\label{Decomposition Lemma}
  Let $\Omega\subset\mathbb{R}^n$ be a bounded open domain  with smooth boundary, let $\Xi_n$ be the integer defined in \eqref{defi of Xi}, and let $j\in \mathbb{Z}_{\geq 0}$, $0<\alpha<1$.  Then, there exist unit vectors $\xi_i\in \mathbb{R}^n$ for $1\leq i \leq \Xi_n$ such that, for every $D\in C^{j,\alpha}(\overline{\Omega},\mathbb{R}_{\mathrm{sym}}^{n\times n})$, there exist a vector field $\Phi\in C^{j+1,\alpha}(\overline{\Omega},\mathbb{R}^n)$ and scalar functions $a_i\in C^{j,\alpha}(\overline{\Omega},\mathbb{R})$ satisfying
\begin{equation}\label{decomp1}
D+\sym\nabla\Phi=\sum_{i=1}^{\Xi_n} a_i^2 \xi_i \otimes \xi_i.
\end{equation}
Furthermore, the following estimate holds for a constant $M>0$ depending only on $j,\alpha,n$ and $\Omega$:
\begin{align}\label{equ-main lemma estimate}
 \lVert \sum_{i=1}^{\Xi_n} a_i^2 \xi_i \otimes \xi_i\rVert_{j,\alpha}+ \lVert  \Phi \rVert_{j+1,\alpha} \leq M\lVert D\rVert_{j,\alpha}.
\end{align}
Moreover, for all $n \notin \{8, 16\}$, $\Xi_n$ is the minimal integer for which such  estimates hold.
 
\end{mainlemm}
Nash \cite{Nash_C1}  decomposed a Riemannian metric as $(g_{ij})=\sum_{i=1}^{n(n+1)/2} a_i^2\xi_i\otimes \xi_i$, which corresponds to setting $\Phi=0$ in the preceding lemma. Gromov discussed such decompositions in \cite[Section 3.8]{Gromov_on_nash}, posing the problem of finding minimal decompositions as Question 60.\par
To prove the 
Decomposition Lemma \ref{Decomposition Lemma}, we construct elliptic systems of size at most $n\times n$ and define $\Phi$ through specific derivatives of their solutions (Lemma \ref{elimination lemma}), thereby eliminating certain rank-one symmetric matrices $a_i^2\xi_i\otimes\xi_i$. The utilization of elliptic systems has the major advantage of avoiding loss of differentiability. 
%This can be compared with, e.g.,  \cite[Theorem 4.2]{DeTurck_Yang_Duke_1984} of Deturck and Yang, where  coordinates are changed to diagonalize $D$  in the $n=3$, $C^\infty$ setting, using moving frames and integration. Along similar lines, in the context of $C^\infty$ isometric embedding, G\"unther \cite{Gunther_ICM_1990} also employed an elliptic operator to avoid loss of differentiability, thereby greatly simplifying the proof and avoiding Nash--Moser iteration.

An interesting aspect of minimizing $\Xi_n$ in the lemma is that it involves classical structures from algebraic geometry and algebraic topology. In fact, the minimality of $\Xi_n$ reflects an algebraic obstruction to further elliptic reduction. These include \textit{projective duality} and the $8$-fold periodicity of the Radon--Hurwitz number, which is related to \textit{Bott periodicity}.  These structures also account for the periodic behavior of $\Xi_n$ in the resulting regularity exponent. We discuss this perspective in Section \ref{introduction decomposition}.

Applying the decomposition lemma within the standard convex integration scheme  leads to the following regularity result, which is a direct consequence of Main Lemma \ref{Decomposition Lemma}.

\begin{theo}\label{main theorem}
Let $n\geq 2$, $0<\beta<1$, $\Omega\subset\mathbb{R}^n$ be a bounded domain with smooth boundary. Given a function $v^\flat\in C^0(\overline{\Omega})$, a vector field $w^\flat\in C^0(\overline{\Omega},\mathbb{R}^n)$ and a matrix field $A\in C^{2,\beta}(\overline{\Omega},\Rsn)$,   for every $\epsilon>0$ and every
\begin{equation*}
0<\alpha<\frac{1}{1+2\Xi_n},
\end{equation*}
there exist $v\in C^{1,\alpha}(\overline{\Omega})$ and $w\in C^{1,\alpha}(\overline{\Omega},\mathbb{R}^n)$ such that 
\begin{equation}
    \begin{gathered}
        \|v-v^\flat\|_0\leq \epsilon, \quad\, \|w-w^\flat\|_0\leq \epsilon,\\
        \frac{1}{2}\nabla v\otimes \nabla v+\sym\nabla w=A.
    \end{gathered}
    \tag{$\ddagger$}
\end{equation}
\end{theo}
The previously known  regularity thresholds for such flexible solutions to \eqref{ddagger} are
\begin{equation*}
     \alpha<\begin{cases}
        \frac{1}{n^2+n+1} &\text{ for } n\geq 3\text{ in  \cite{lewicka2025monge}}, \\
        \frac{1}{3}& \text{ for } n=2 \text{ in \cite{cao_hirsch_inauen_2025}}.
    \end{cases}
\end{equation*} 
Thus, our results improve the known regularity for every $n\geq 3$. In fact, the same reduction in the number of rank-one symmetric matrices is
compatible with the standard convex integration scheme for \eqref{NK}; see Remark \ref{rema-on NK} and Remark \ref{rema-on recent Cao} for a discussion.

Equation \eqref{ddagger} moreover attracts interest due to its relation to the \textit{very weak solutions} of several nonlinear equations, a connection recently observed in several works.  In their seminal work \cite{lewicka_convex_2017}, Lewicka and Pakzad first noticed that in dimension $n=2$, by applying $curl\,curl$ to both sides, \eqref{ddagger} is reduced to \textit{Monge--Amp\`ere equation} (see also \cite{Daniel_Hornung}):
\begin{equation}
    \mathfrak{Det}\nabla^2v:=-\frac{1}{2}curl \ curl(\nabla v\otimes\nabla v)=f.
\end{equation}
 Subsequently, it was followed by \cite{cao_very_2019,cao_hirsch_inauen_2025} to improve the regularity of the solutions when $n=2$. Equation \eqref{ddagger} was further studied in various settings, including higher-dimensional ones, by Lewicka in \cite{lewicka2025-2monge,lewicka2025monge,lewicka2025-3monge}, by Inauen and Lewicka in \cite{inauen2025monge}, and by Cao, Hirsch, Inauen, and Lewicka in \cite{cao2026full}.
In particular, it is observed that, when $n=2$, the left-hand side of \eqref{ddagger} corresponds to the \textit{von K\'arm\'an content}; see also \cite{Daniel_Hornung}, where \eqref{ddagger} is treated in dimension $2$ as the \textit{von K\'arm\'an system}.  Applying  $\mathfrak{C}^2$ on both sides, where $\mathfrak{C}^2(A)_{ij,st}:=\partial_i\partial_sA_{jt}+\partial_j\partial_tA_{is}-\partial_i\partial_tA_{js}-\partial_j\partial_sA_{it} $ for any $A:\Omega\rightarrow \Rsn$,
 one obtains what is referred to as the \textit{Monge--Amp\`ere systems}: 
\begin{equation}
    \mathfrak{Det}\nabla^2v:=[\partial_i\partial_jv\cdot\partial_s\partial_tv-\partial_i\partial_tv\cdot\partial_j\partial_sv]_{ij,st:1\cdots n}=-\mathfrak{C}^2(A),
\end{equation}
which is shown to be equivalent to the problem \eqref{ddagger}, up to regularity issues, as discussed in \cite[Section 1.3]{lewicka2025monge}. In \cite{li_Qiu_very_2024}, Li and Qiu applied the operator $\mathfrak{L}(A):=\sum_{i,j}\partial_i\partial_iA_{jj}+\partial_j\partial_jA_{ii}-2\partial_i\partial_jA_{ij}$, defined for any $A:\Omega\rightarrow \Rsn$, to both sides of \eqref{ddagger} in order to study the $2$-Hessian equation in arbitrary dimension $n$ (see also \cite{dinew_very_2024}): 
\begin{equation}
\sigma_2(\nabla^2v):=\sum_{i,j=1}^n[\partial_i\partial_iv\cdot \partial_j\partial_jv-\partial_i\partial_jv\cdot \partial_i\partial_jv]=f.
\end{equation}
Cao and Wang used a similar strategy in \cite{CaoWang2026} to relate \eqref{ddagger} to the two-dimensional Lagrangian mean curvature equation:
\begin{equation}
    curl\ curl (\frac{1}{2}\nabla v\otimes \nabla v+\sym\nabla w-(v \cot \Theta)Id+VId)=-1,
\end{equation}
 where $VId$ is a lower-order error term and $\Theta:\overline\Omega\to(-\pi,\pi)$ is the phase function $\Theta: \overline{\Omega}\rightarrow (-\pi,\pi)$. The solutions to the above equations are very weak in the distributional sense. The improvement in the regularity of $v\in C^{1,\alpha}$ in Theorem \ref{main theorem}
   leads to a corresponding improvement in the regularity of these very weak solutions.\par

\subsection{Convex integration: A review}
We briefly review the convex integration scheme here. The scheme of the iteration was introduced by Nash in \cite{Nash_C1}, modified 
by Conti-De Lellis-Sz\'ekelyhidi in  \cite{conti_h-principle_2012}, and was adapted to equation \eqref{ddagger} in \cite{lewicka_convex_2017} and subsequent works.\par

The general scheme goes as follows. The solution of \eqref{ddagger} is approached by $(V_q,W_q)$ that is constructed by iteration, and the passage from $(V_q,W_q)$ to $(V_{q+1},W_{q+1})$ is called \textbf{one stage}. In each stage, one first mollifies $(V_q,W_q)$, and denotes the deficit in stage $q$ by 
\begin{equation*}
     D_q:=A-\frac{1}{2} \nabla V_q\otimes \nabla V_q-\sym\nabla W_q.
\end{equation*}
The goal of each stage iteration is to make up for the deficit. To achieve that, we first decompose $D_q$ into $M$ rank-one symmetric matrices
\begin{equation}\label{decomposition on M}
    D_q+\sym\nabla \Phi=\sum_{i=1}^M a_i^2\xi_i\otimes\xi_i
\end{equation}
for some $M$; notice that $M$ can  always be taken as
$\frac{n(n+1)}{2}$ with $\Phi=0$ due to Nash \cite{Nash_C1}. Then we divide one stage into \textbf{$M$ steps} $(v_0,w_0),(v_1,w_1),\cdots,(v_M,w_M)$ according to the decomposition, with 
\begin{enumerate}
    \item the initial step $(v_0,w_0)$ being the mollification of $(V_q,W_q)$,
    \item each $ (v_{i},w_{i})$ constructed from $(v_{i-1},w_{i-1})$ 
 by adding specific corrugation functions with small amplitudes (decreasing with respect to stage $q$) and large frequencies (increasing with respect to stage $q$), 
    \item the corrugation in the $i$-th step designed to correct a single term $a_i^2\xi_i\otimes\xi_i$, the term $\Phi$ being absorbed into $w_1$ in the first step, and
    \item the final step $(v_M,w_M)$ set as the next stage $(V_{q+1},W_{q+1})$.
\end{enumerate}
Taking $\delta_q\rightarrow0$ and showing $D_q\rightarrow 0$, a solution of \eqref{ddagger} is obtained as required in Theorem \ref{main theorem}. 
 Crucially, as the proof demonstrates,
 minimizing the number of steps $M$ under suitable estimates directly improves the regularity of the resulting solutions.\par

\begin{rema}\label{rema-on NK}
The Decomposition Lemma \ref{Decomposition Lemma} can be applied within the standard convex integration framework for \eqref{NK} from \cite{conti_h-principle_2012} (cf. \cite{cao2025nashkuipertheoremisometricimmersions}, \cite{DeLellis_Szekelyhidi_on_nash}). 
While we do not explicitly perform the iteration here, the estimates in Lemma \ref{Decomposition Lemma} are compatible with a reduction in the number of steps per stage, which in principle can lead to improving Hölder regularity of solutions to $\alpha < \frac{1}{1+2\Xi_n}$. Related ideas are present in the convex integration frameworks for \eqref{NK};
see, for example, \cite[Section 5.2]{DeLellis_Szekelyhidi_on_nash}. Our formulation
of this point was guided in particular by  \cite[Lemma 3.1, Propositions 3.2 and 3.4]{cao2025nashkuipertheoremisometricimmersions}.
\end{rema}

\subsection{Main ideas of the decomposition lemma}\label{introduction decomposition}

One useful and less standard  point of view in this paper is to use classical algebraic geometry to reduce the number of rank-one symmetric matrices in the decomposition, after allowing the additional term $\sym\nabla\Phi$. To obtain the decomposition in the form of Lemma \ref{Decomposition Lemma}, using the theory of elliptic systems (see Section \ref{preliminaries elliptic}), we must establish: 
\begin{enumerate}
    \item A linear subspace $L \subset \mathbb{R}^{n\times n}_{\mathrm{sym}}$ whose nonzero elements $A$ satisfy that $A+\diag(A)$, obtained by doubling the diagonal entries of $A$, is invertible. This condition enables the construction of an elliptic system for $\Phi$; solving this system ensures that $D+\sym\nabla \Phi$ lies in  $L^\perp$ (see Lemma \ref{elimination lemma});
    \item The assurance that $D+\sym\nabla\Phi$ (or more precisely, $L^\perp$) can be spanned by  $\frac{n(n+1)}{2}-\dim L$ many rank-one symmetric matrices $\xi_1\otimes\xi_1,\cdots,\xi_{\frac{n(n+1)}{2}-\dim L}\otimes\xi_{\frac{n(n+1)}{2}-\dim L}$ (see Lemma \ref{nonnegative coefficient lemma}).
\end{enumerate}
Our objective is to determine the maximal dimension of $L$. This directly minimizes the number of steps $M$ in \eqref{decomposition on M}, thereby yielding the highest regularity achievable by our method.\par
It is natural to projectivize this setup, reformulating the two conditions as \textit{intersection problems} in $\PRs$:
\begin{enumerate}
   \item $\bP(L)\cap \bY=\varnothing$, where $\bY \subset \PRs$ is the hypersurface of matrices whose determinants vanish after scaling the diagonal by $2$; and
\item The Veronese image of $\RP^{n-1}$ in $\PRs^\vee$, denoted by $\bZ$, intersects $\bP(L^\perp)$ in such a way that the intersection spans $\bP(L^\perp)$.
\end{enumerate}
A key observation is that these two requirements are closely coupled: the varieties $\bY$ and $\bZ$ are dual to each other in the sense of \textit{projective duality} (see Section \ref{sec-projective duality}).\par
Notably, the above first condition exhibits a clear $8$-fold periodicity, which was studied by Adams, Lax, and Phillips in the 1960s. In \cite{ALP_1965,ALP_1966_correction}, they determined that the maximum dimension of a real vector space of symmetric $n\times n$ matrices, in which every nonzero element is invertible, is $\rho(\frac{1}{2}n)+1$. This result relies on Adams' theorem concerning the maximum number of linearly independent vector fields on spheres \cite{Adams_1962_vectorfields}.

\begin{theo}[Adams \cite{Adams_1962_vectorfields}]
       The maximum number of linearly independent vector fields on $S^{n-1}$ is $\rho(n)-1$.
\end{theo}
We are now in the position to emphasize that Adams' theorem can be understood through the $K$-theory of real projective spaces and Bott periodicity, which explains the periodicity of $\rho(n)$ and, in turn, of the index $\Xi_n$. Thus, the algebraic obstruction to further elliptic reduction in the convex integration framework is ultimately connected with the  topology underlying the Radon--Hurwitz number.
\par

Sections \ref{subsection-algebraic description} and \ref{section of intersections} adapt these results from algebraic geometry and topology to our setting. We first construct subspaces $L$ with $\bP(L)\cap\bY=\varnothing$ and $\bZ\cap\bP(L^\perp)\neq\varnothing$; a perturbation argument then produces nearby subspaces for which the latter intersection becomes nondegenerate. This yields subspaces satisfying both required conditions, of dimension $\rho(\frac12 n)+1$ for all $n\notin \{2,4,8,16\}$. We state the resulting algebraic theorem in the following, leaving the remaining optimality question in dimensions $n=8,16$ as a conjecture.

%For these four exceptional dimensions, one can construct such a subspace $L$ with one dimension less, namely $\dim L=\rho(\frac{1}{2}n)=\frac{1}{2}n$. It is known (see Proposition \ref{n=2} and \ref{n=4}) that this reduced 1 dimension is in fact  maximum for $n=2,4$, while for $n=8,16$, the optimal dimension remains undetermined---it may lie between  $\rho(\frac{1}{2}n)$ and $\rho(\frac{1}{2}n)+1$.\par

\begin{theo}\label{second theorem}
    Let $L\subset\Rsn$, $n\geq 2$ be a subspace satisfying
    \begin{enumerate}
        \item $\bP(L)\cap \bY=\varnothing$, \  and
        \item $\bP(L^\perp)\cap\bZ$ is nonempty and spans $\bP(L^\perp)$.
    \end{enumerate}
   Then the maximum possible dimension of such an $L$ is 
   \begin{equation*}
       \begin{cases}
       \rho(\frac{1}{2}n)  & n=2,4,\\
       \rho(\frac{1}{2}n)+1 &  n\in\mathbb{Z}_{\geq2}-\{2,4,8,16\}.
    \end{cases}
   \end{equation*} 
   For $n=8,16$, the maximal possible dimension lies between $\rho(\frac{1}{2} n)$ and $\rho(\frac{1}{2} n)+1$.
\end{theo}
For  dimensions $n=2,4,8,16$, the existence of normed division algebras $\mathbb{R}$, $\mathbb{C}$, $\mathbb{H}$, and $\mathbb{O}$ provides subspaces $L$ of dimension $\rho(\frac{1}{2}n)+1=\frac{1}{2}n+1$ with $\bP(L)\cap \bY = \varnothing$. However, such large subspaces appear to be incompatible with a nonempty intersection $\bP(L^\perp)\cap \bZ$.  \par
With the $n=2,4$ cases known, see Propositions \ref{n=2}, \ref{n=4}, we propose the following conjecture for the unsettled cases $n=8,16$, which may also be of independent interest.
\begin{conj}[Quadrics base locus conjecture \ref{conjecture}]
    For $n=8,16$, the maximal  dimension of $L$ in Theorem \ref{second theorem} is $\rho(\frac{1}{2}n)$.
\end{conj}

\begin{rema}
    The idea of using $\sym\nabla\Phi$ to reduce the number of terms in the decomposition extends  \cite[Proposition 3.1]{cao_very_2019} by Cao and Sz\'ekelyhidi  from dimension $2$ to arbitrary $n$. In \cite{cao_very_2019}, the authors interpret this construction in terms of a \textit{planar div-curl system}.
\end{rema}

\begin{rema}\label{rema-on recent Cao}
   Very recently, when we were about to finish this paper, we learned of the breakthrough by Cao, Hirsch, and Inauen in \cite{cao2025nashkuipertheoremisometricimmersions}, which improves the regularity threshold for solutions of \eqref{NK} to $\alpha<\frac{1}{n^2-n+1}$ for the H\"older exponent. Their result is achieved by employing novel \textit{corrugation ansatz} and \textit{integration by parts}   in  convex integration, whereas our approach uses  an alternative  \textit{reduction of matrix decomposition} in strategy. As the authors remark in \cite{cao2025nashkuipertheoremisometricimmersions}, their theorem is  adaptable to the \eqref{ddagger} problem. Therefore, future work may explore combining these techniques to further improve the regularity of the solutions to \eqref{ddagger}, and extend the study of \eqref{NK}.
\end{rema}

\begin{ack} 
The authors thank Sergey Galkin, Runjian Huang, Dominik Inauen, Ilia Itenberg,  Hua-Zhong Ke, Changzheng Li, Tongtong Li, Xiangfei Li, Jiayu Song, Lei Song, Ju Tan,  Xiaowei Wang, Hui-chun Zhang and Jintian Zhu, with whom we discussed various aspects of this project. We would also like to thank Wentao Cao for discussions and suggestions in our early draft. Special thanks are given to our friend Tongtong Li, who worked closely with us at the beginning of this project.  The authors would have preferred that he join us as a coauthor of this paper, but have to respect his wishes in this regard.
Special thanks also go to Sergey Galkin, who generously shares with us the proof of Proposition \ref{n=4} via Euler characteristics. Z. Su thanks Sergey Galkin, Changzheng Li and Xiaowei Wang for their encouragement and sustained interest in this research. W. Zhang thanks Hui-chun Zhang for his gentle encouragement and kind help in this research.

\end{ack}

\section{Preliminaries}

\subsection{Preparations on iterations of convex integration}

For general analysis notations, let $C^{k,\alpha}(\overline{\Omega},\mathbb{R}^m)$ denote the standard H\"{o}lder space. We denote the $C^k$ norm and the $C^{k,\alpha}$ norm as $\lVert \,\cdot\, \rVert_{k}$ and $\|\cdot\, \rVert_{k,\alpha}$, respectively.
When there is no ambiguity, we also write $|\cdot|$ for $\|\cdot\|_0$. 
For brevity, we set $\|(v,w)\|_k:=\|v\|_k+\|w\|_k.$

\subsubsection{Corrugation functions}
Following \cite{lewicka_convex_2017}  and subsequent related works on \eqref{ddagger},  we will use the following \textbf{corrugation functions} that were originally from Kuiper \cite{Kuiper1955OnC1_I&II}.   Denote 
\begin{equation*}
    \Gamma_1(s,t):=\frac{s}{\pi}\sin (2\pi t),\quad \Gamma_2(s,t):=-\frac{s^2}{4\pi}\sin (4\pi t). 
\end{equation*}
Then they satisfy the identity
\begin{equation}\label{dagger_Gamma}
\partial_t\Gamma_2(s,t)+\frac{1}{2}|\partial_t \Gamma_1(s,t)|^2=s^2.
     \tag{$\dagger_ \Gamma $}
\end{equation}
As a consequence of definitions, we have the following estimates for $0\leq k \leq 3$  and some $C>0$:
\begin{equation}\label{Gamma estimates}
  \begin{array}{lll}
|\partial_t^k\Gamma_1(s,t)|\leq C|s|, & |\partial_s\partial_t^k\Gamma_1(s,t)|\leq C,
&
\\
|\partial_t^k\Gamma_2(s,t)|\leq Cs^2, 
& |\partial_s\partial_t^k\Gamma_2(s,t)|\leq C|s|,
&
|\partial_s^2\partial_t^k\Gamma_2(s,t)|\leq C,
  \end{array} \quad \text{ for any }s, t\in \RR.
\end{equation}
Such functions $\Gamma_1,\Gamma_2$ are important for convex integration iteration. 

\subsubsection{Mollification}
Let $\ast$ denote convolution.  Let $\phi_l$ denote the standard mollifier at scale $l>0$. For any $f\in C^{k,\alpha}(\overline{\Omega},\mathbb{R}^m)$, to define $f*\phi_l$ as a map on $\overline{\Omega}$, we apply the Whitney extension theorem to extend $f$ to a map $\mathring{f}\in C^{k,\alpha}(\mathbb{R}^n,\mathbb{R}^m)$, where the regularity of  boundary of $\Omega$ is required. The mollification is then given by $f*\phi_l:=(\mathring{f}*\phi_l)|_{\overline{\Omega}}$ (see also \cite[Section 2.2]{cao_hirsch_inauen_2025}). We recall here some basic yet very useful inequalities, the proofs and more details of mollification can be found in \cite{conti_h-principle_2012}.
\begin{lemm}[Mollification Lemma]\label{Mollification Lemma}
    For any $0<\alpha<\beta\leq1$, $0\leq r\leq s\leq 2$, and $k\in\mathbb{Z}_{\geq0}\,,$ $f,f_1,f_2\in C^{k,\alpha}(\overline{\Omega})$, we have
    \begin{align}
        &[f]_{\alpha}\leq C\|f\|_0^{1-\frac{\alpha}{\beta}}[f]_{\beta}^{\frac{\alpha}{\beta}},\\
        &\|f_1f_2\|_{k,\alpha}\leq C\|f_1\|_{k,\alpha}\|f_2\|_{k,\alpha},\\
        &[f\ast\phi_l]_{r}\leq C[f]_r,\\
        &[f\ast\phi_l]_{r+s}\leq Cl^{-s}[f]_r,\\
        &[f-f\ast\phi_l]_{r}\leq Cl^{s-r}[f]_{s},\label{f-f*phi}\\
        &[(f_1f_2)\ast\phi_l-(f_1\ast\phi_l)(f_2\ast\phi_l)]_r\leq Cl^{2s-r}[f_1]_s[f_2]_s, \text{ for }0<s\leq 1.
    \end{align}
    where $\phi_l$ is the standard mollification kernel with scale $l>0$.
\end{lemm}

\subsection{Elliptic theory to linear systems}\label{preliminaries elliptic}
For later applications, we need some elliptic regularity theory for linear systems. For coupled elliptic systems, unlike the scalar case, we need the following ellipticity conditions.
\begin{defi}\label{elliptic conditions}
    For a system of $m$ equations over domains in $\RR^n$, the matrix of coefficients $(A^{\alpha\beta}_{ij})^{1\leq\alpha,\beta\leq n}_{1\leq i,j\leq m}$ is said to satisfy
\begin{itemize}
    
\item the very strong ellipticity condition, or the \textbf{Legendre condition}, if
there is a $\lambda>0$ such that
\begin{equation}
\label{eq:L-cond}
A^{\alpha\beta}_{ij}\xi^i_\alpha\xi^j_\beta\geq\lambda|\xi|^2, \forall\xi\in\RR^{m\times n},
\end{equation}
\item the strong ellipticity condition, or the \textbf{Legendre-Hadamard condition}, if there is a $\lambda>0$ such that
\begin{equation}
\label{eq:L-H-cond}
A^{\alpha\beta}_{ij}\xi_\alpha\xi_\beta\eta^i\eta^j\geq\lambda|\xi|^2|\eta|^2, \forall\xi\in\RR^n,\eta\in\RR^m.
\end{equation}
\end{itemize}
The constant $\lambda$ is called an ellipticity constant.
\end{defi}

\begin{rema}
The Legendre condition is stronger than the Legendre-Hadamard
condition. Indeed, one 
just takes $\xi^i_\alpha$ as $\xi_\alpha\eta^i$. Note that the converse is trivially true
in case $m=1$ or $n=1$, but is not true in general.
\end{rema}

Next we focus on the following elliptic system:
\begin{equation}
    \label{eqs:ells}
    \begin{cases}
        -D_\beta (A^{\alpha\beta}_{ij}D_\alpha u^i )=f^j \qquad&\text{ in }\Omega,\\
        u^i=0\qquad&\text{ on }\partial\Omega,
    \end{cases}
\end{equation}
where $u:=\{u^i\}_{i=1,\dots,m}, f:=\{f^j\}_{j=1, \dots,m}$ are vector functions over domains $\Omega\subset \RR^n$  to $\RR^m$.

It can be  checked that the Dirichlet problem of \eqref{eqs:ells} is  solvable in $W^{1,2}(\Omega)$ by Lax-Milgram Theorem, under the Legendre condition, or the Legendre-Hadamard condition with constant coefficients (cf. theorem 3.42 in \cite{giaquinta_introduction_2012}). Therefore, we can move forward to the improvement of the regularity to the existing weak solutions. 

In the cases considered below, the coefficients are constant, then by Theorem 4.14 in \cite{giaquinta_introduction_2012}, weak solutions are upgraded to $W^{2,2}(\Omega)$. Thanks to the linearity, the Dirichlet problem of \eqref{eqs:ells} has at most one solution. Since the coefficients are constant, the system can also be written in non-divergence form, hence by theorem 5.25 in \cite{giaquinta_introduction_2012}, we have the following theorem.

\begin{theo}\label{elliptic theorem}
    The Dirichlet problem for the following coupled elliptic system
    \begin{equation*}
        \begin{cases}
        -D_\beta (A^{\alpha\beta}_{ij}D_\alpha u^i )=f^j \qquad&\text{in}~\Omega,\\
        u^i=0\qquad&\text{on}~\partial\Omega,
        \end{cases}
    \end{equation*}	    	 
    where $\Omega$ is a bounded domain in $\RR^n$ with $\partial\Omega\in C^{2,1}$, $\{f^j\}_{j=1\ \dots,m}\in C^{k,\gamma}(\overline{\Omega})$, the coefficients $A^{\alpha\beta}_{ij}$ are constants and satisfy \eqref{eq:L-H-cond}, has a unique solution $\{u^i\}_{i=1,\dots,m}\in C^{k+2,\gamma}(\overline{\Omega})$. Furthermore, we have 
    \begin{equation}
        \|u\|_{k+2,\gamma}\leq C\|f\|_{k,\gamma},
    \end{equation}
    where $C$ depends only on $\Omega, k, \gamma$ and the ellipticity constants $\lambda$.
\end{theo}

\subsection{Radon--Hurwitz number: an introduction}\label{RH number: a review}
Hurwitz's theorem states that the only normed division algebras are reals $\mathbb{R}$, complexes $\mathbb{C}$, quaternions $\mathbb{H}$ and octonions $\mathbb{O}$ \cite{Lee_1948_Clifford}.
This classical result is closely related to Clifford algebras and Bott periodicity, which underlies an $8$-fold periodicity in topology and geometry \cite{Atiyah-Bott-Shapiro_1964, Bott_1959_periodicity}.\par

 For convenience, and with a slight abuse of terminology, we say a vector space  $W\subset\mathbb{R}^{n\times n}$ (resp. $W\subset\mathbb{R}^{n\times n}_{\mathrm{sym}}$) is \textbf{invertible} if every nonzero element of $W$ is invertible.  Recall  that the $\ell$ vector fields $v_1,\cdots,v_\ell$ on a manifold are said to be \textbf{independent} if for each point $x$ on the manifold, the vectors $v_1(x),\cdots,v_\ell(x)$ are linearly independent. Additionally, for definition and classification of representations of Clifford algebras $\mathrm{Cl}_{n}$, see \cite{lawsonMichelsohn_2016_spin}. 
\begin{prop}\label{prop-three statements imply}
    Consider the following existence statements on $\ell$.
\begin{enumerate}[label=\textcolor{blue}{\roman*.}, ref=\roman*]\label{Radon--Hurwitz existence prop}
  \item The space of $n \times n$ real matrices has an $\ell$-dimensional invertible subspace $W_n$. \label{one}   \item $\mathbb{R}^n$ admits a $\mathrm{Cl}_{\ell-1}$ representation.\label{two}
    \item There are $\ell-1$ independent vector fields on sphere $S^{n-1}$. \label{three}
\end{enumerate}

Then \eqref{one} implies \eqref{three}, and \eqref{two} implies \eqref{three}.
\end{prop}
\begin{proof}
 \eqref{two} implies \eqref{three} as proved in \cite[Theorem 7.1, Chapter 1]{lawsonMichelsohn_2016_spin}. To see \eqref{one} implies \eqref{three}, we define the $\ell-1$ vector fields at each unit vector $\varepsilon\in S^{n-1}\subset\mathbb{R}^n$ to be $\Pi_{\varepsilon^\perp}(A_1^{-1}A_i\varepsilon):=A_1^{-1}A_i\varepsilon -\langle A_1^{-1}A_i\varepsilon,\varepsilon\rangle\varepsilon$ for $2\leq i\leq \ell$, where $\{A_1,\cdots,A_\ell\}$ is a basis of the invertible $W_n$, and $\langle\ \ ,\ \ \rangle$ is the standard inner product on $\mathbb{R}^n$. Assume that these vector fields are not independent, namely 
 for some $(c_2,\cdots,c_{\ell})\neq (0,\cdots,0)$ and some $\varepsilon_0\in S^{n-1}$,  one has
 $\sum_{i\geq 2}c_i\Pi_{\varepsilon_0^{}\perp}(A_1^{-1}A_i\varepsilon_0)=0$, which is equivalent to
 \begin{equation}\sum_{i\geq 2} c_i(A_i\varepsilon_0)-(\sum_{i\geq 2} c_i \langle A_1^{-1}A_i\varepsilon_0,\varepsilon_0\rangle)A_1\varepsilon_0=\left(\sum_{i\geq 2} c_iA_i-(\sum_{i\geq 2} c_i \langle A_1^{-1}A_i\varepsilon_0,\varepsilon_0\rangle)A_1\right)\varepsilon_0=0.
 \end{equation}
    It contradicts the assumption  that every nonzero linear combination of $A_i$ is invertible.
\end{proof}

The study of independent vector fields on spheres is closely related to the present work as the above proposition shows. Classical constructions of Radon and Hurwitz produce $\rho(n)-1$ independent vector fields on $S^{n-1}$, where $\rho(n)$ is the Radon--Hurwitz number (see Definition \ref{R-H number defi}). Adams \cite{Adams_1962_vectorfields} showed that this bound is optimal by using methods from homotopy theory and topological $K$-theory, in connection with Bott periodicity.  See also \cite{Eckmann_1994_HurwitzRadonMR} for related discussions and further background on this relation.

\begin{theo}[Adams' theorem on vector fields on spheres]\label{adams theorem}  There exist $\rho(n)-1$ independent vector fields on $S^{n-1}$, and there do not exist $\rho(n)$ such vector fields.
\end{theo}
Note for odd $n$, this recovers the hairy ball theorem for the even-dimensional sphere $S^{n-1}$. The existence of $\rho(n)-1$ independent vector fields can be obtained from the classical constructions of Radon and Hurwitz; see also Proposition \ref{prop-construction of W_n}. For a proof of existence via representations of Clifford algebras, we refer to \cite{Atiyah-Bott-Shapiro_1964} (cf. \cite[Theorem 7.2, Chapter 1]{lawsonMichelsohn_2016_spin}).

In the following, we use the standard notations $\oplus$ and $\otimes$ for the Kronecker sum and Kronecker product of matrices, respectively, as in \cite[Chapter~4]{Horn_Johnson_1991_matrix}.

\begin{prop}\label{prop-construction of W_n}
    For any $\lambda\in\Lambda=\mathbb R, \mathbb C, \mathbb H, \mathbb O$, denote \(\mathsf L_\lambda\) as the real \(\dim \Lambda\times \dim\Lambda\) matrix representing left multiplication by $\lambda$. Let $W_n$ be defined inductively as the following.
    \begin{enumerate}
        \item  If $n=1,2,4,8$, let
   $  W_n = \{\mathsf L_\lambda  \mid \lambda \in \Lambda\}$,
   where \( \Lambda = \mathbb R, \mathbb C, \mathbb H, \mathbb O\), respectively.

          \item If $n = 2^{4a}$ with $a\ge 1$, let        $W_{n} = \left\{ \begin{bmatrix} rI_{2^{4a-1}} & A \\ A^T & -r I_{2^{4a-1}} \end{bmatrix} \,\middle|\, r\in \mathbb R,\, A\in W_{2^{4a-1}} \right\}$.

    \item If $n = 2^{4a+1}, 2^{4a+2}, 2^{4a+3}$ with $a\ge 1$, let         $W_n = \{\, A \oplus \mathsf L_\iota := A \otimes I_{\dim_\mathbb{R}\Lambda} + I_{2^{4a}} \otimes \mathsf L_\iota \, \mid \, A\in W_{2^{4a}}, \iota \in \mathrm{Im}(\Lambda) \,\}$, where $\Lambda = \mathbb C, \mathbb H, \mathbb O$, respectively.

        \item For general $n=2^{4a+b}(2c+1)$ with $a,b,c\geq 0$, let $W_n=\{A\otimes I_{2c+1}\,|\, A\in W_{2^{4a+b}}\}$.
    \end{enumerate}
    Then $W_n$ is invertible, and $\dim W_n=\rho(n)$.
\end{prop}
\begin{proof}
    We only check Case 3, as the other cases are immediate. Let $A\in W_{2^{4a}}$ and $\iota \in \mathrm{Im}(\Lambda)$, and assume $(A,\iota)\neq (0,0)$. 
    For any $A\oplus \mathsf{L}_\iota\in W_n$ with  $A\in W_{2^{4a}}$ and $\iota \in \mathrm{Im}(\Lambda)$, $\Lambda=\mathbb{C},\mathbb{H},\mathbb{O}$, we  assume $(A,\iota)\neq (0,0)$. As a basic property of the Kronecker sum, each eigenvalue of
$A\oplus \mathsf{L}_\iota$ is the sum of an eigenvalue of $A$ and an eigenvalue of $\mathsf{L}_\iota$.  Since the eigenvalues of
$\mathsf{L}_\iota$ are purely imaginary, while the eigenvalues of $A\in W_{2^{4a}}$ are real because $A$ is symmetric, all
eigenvalues of $A\oplus \mathsf{L}_\iota$ are nonzero. Hence $A\oplus \mathsf{L}_\iota$ is invertible. It follows directly from the construction that $\dim W_n = \rho(n)$.

\end{proof}

\begin{rema}
As Adams explains in \cite{Adams_1962_vectorfields}, the depth of Theorem \ref{adams theorem}, concerning the nonexistence of $\rho(n)$ independent vector fields on $S^{n-1}$, increases with $n$. By Proposition \ref{prop-three statements imply}, this yields the bound $\dim W_n \leq \rho(n)$ for any invertible vector space $W_n$ of $n \times n$ real matrices.\par

While Adams’ original proof of the maximal number of independent vector fields is technically involved, the same dimensional bound for invertible matrix spaces can also be obtained through approaches that do not require the full technical machinery of Adams'. In particular, as observed in \cite[Theorem 12]{petrovic_1996_spaces}, one may use the structure of the real $K$-group $KO(\mathbb{R}\boldsymbol{\mathrm{P}}^n)$ to derive $\dim W_n \leq \rho(n)$, where Bott periodicity is naturally encoded in $KO(\mathbb{R}\boldsymbol{\mathrm{P}}^n)$.
\end{rema}

\section{Decomposition lemma}\label{section of decomposition lemma}
In this section, one finds the key observation of this paper: by solving an elliptic system,  one can find a $\Phi$ such that $D-\frac{1}{2}(\nabla \Phi+(\nabla \Phi)^T)$ decomposes into fewer than $\frac{n(n+1)}{2}$ rank-one symmetric matrices for every $n\geq2$.

\subsection{Elimination lemma and nonnegative coefficient lemma}\label{section: two lemmas}

Recall we equip the space of symmetric matrices $ \mathbb{R}_{\mathrm{sym}}^{n\times n}$ with the Euclidean inner product 
\begin{equation}
    \langle A', A'' \rangle_{\mathbb{R}_{\mathrm{sym}}^{n\times n}} := \sum_{1 \le i \le j \le n} a_{ij}' a_{ij}'', \quad \text{for } A'=(a'_{ij}), A''=(a''_{ij}).
\end{equation}
This  defines the orthogonal subspace $L^\perp:=\{A'\in \mathbb{R}_{\mathrm{sym}}^{n\times n}\mid\,\langle A',A''\rangle_{\mathbb{R}^{n\times n}_{\mathrm{sym}}}=0,\, \forall\, A''\in L\}$ of $L \subset \mathbb{R}^{n\times n}_{\mathrm{sym}}$ and the orthogonal projection $\Pi_{L} : C^{j,\alpha}(\overline{\Omega},\mathbb{R}^{n\times n}_{\mathrm{sym}}) \longrightarrow C^{j,\alpha}(\overline{\Omega},L)$. 

\begin{lemm}[elimination lemma]\label{elimination lemma}
    Let $D\in C^{j,\alpha}(\overline{\Omega},\mathbb{R}_{\mathrm{sym}}^{n\times n})$. Suppose $L\subset \mathbb{R}_{\mathrm{sym}}^{n\times n}$ is  a subspace such that for every non-zero matrix $A \in L$, the matrix $A + \mathrm{diag}(A)$ is invertible.  Then there exists $\check{\Phi} \in C^{j+1,\alpha}(\overline{\Omega}, \mathbb{R}^n)$ such that
    \begin{equation}\label{elimination}
        \Pi_L\left(D+\sym\nabla\check{\Phi}\right)=0\quad \text{ in }\Omega.
    \end{equation}
\end{lemm}

\begin{proof}

Denote the formal adjoint of the operator $\sym\nabla: C^{j+1,\alpha}(\overline{\Omega},\mathbb{R}^n) \longrightarrow C^{j,\alpha}(\overline{\Omega},\mathbb{R}^{n\times n}_{\mathrm{sym}})$ by $(\sym\nabla)^*$, taken with respect to the $L^2$-inner products on $\mathbb{R}^n$ and $\mathbb{R}^{n\times n}_{\mathrm{sym}}$. We first compute its principal symbol $\sigma((\sym\nabla)^*)(\xi)A$ for any frequency vector $\xi=[\xi_1,\cdots,\xi_n]^T\in \mathbb{R}^n$ and matrix $A=(a_{ij})\in \mathbb{R}_{\mathrm{sym}}^{n\times n}$. 

Using the identification $\mathbb{R}^{n\times n}_{\mathrm{sym}}\cong \mathbb{R}^{\frac{n(n+1)}{2}}$ given by $A=(a_{ij})\mapsto [a_{12},a_{13},\cdots,a_{n-1,n},a_{11},\cdots,a_{nn}]^T$, we can express the principal symbol of the symmetric gradient as $\sigma(\sym\nabla)(\xi) = i B(\xi)$, where the matrix $B(\xi):\mathbb{R}^n\rightarrow\mathbb{R}^{\frac{n(n+1)}{2}}$ is defined as
\begin{equation*}
    B(\xi)=\frac{1}{2}\begin{bmatrix}
    \xi_2& \xi_1& 0&\cdots & 0\\
    \xi_3&0&\xi_1&\cdots&0\\
    \vdots&\vdots&\vdots&\ddots&\vdots\\
    0&0&\cdots&\xi_{n}&\xi_{n-1}\\ \hdashline[2pt/2pt]
    2\xi_1 &0&\cdots&0&0\\
    \vdots&\vdots&\vdots&\ddots&\vdots\\
    0&0&\cdots&0&2\xi_n
    \end{bmatrix}.
\end{equation*}

By taking the conjugate transpose, we find that the symbol of the adjoint operator is given by
\begin{equation*}
    \sigma((\sym\nabla)^*)(\xi) A = -iB(\xi)^T A = -\frac{i}{2} \begin{bmatrix}
      2a_{11}\xi_1+a_{12}\xi_2+\cdots+a_{1n}\xi_n \\
      a_{12}\xi_1+2a_{22}\xi_2+\cdots+a_{2n}\xi_n \\
      \vdots\\
      a_{1n}\xi_1+a_{2n}\xi_2+\cdots+2a_{nn}\xi_n
     \end{bmatrix} = -\frac{i}{2}\left(A+\diag(A)\right)\xi.
\end{equation*}

To solve \eqref{elimination}, we make the ansatz $\check{\Phi} = (\sym\nabla)^*\Pi_L^* u$ for a section $u \in C^{j+2,\alpha}(\overline{\Omega}, L)$, with $\Pi_L^*:C^{j,\alpha}(\overline{\Omega},L)\longrightarrow C^{j,\alpha}(\overline{\Omega},\mathbb{R}^{n\times n}_{\mathrm{sym}})$ being the inclusion map. Substituting this into  \eqref{elimination},  it is enough to solve the system
\begin{equation}\label{equ-elliptic system for u_i}
    \begin{cases} 
    \Pi_L \sym\nabla (\sym\nabla)^*\Pi_L^* u = -\Pi_L D \quad &\text{in } \Omega, \\
    u = 0 &\text{on } \partial\Omega.
    \end{cases}
\end{equation}
We now verify that $ \Pi_L \sym\nabla (\sym\nabla)^*\Pi_L^*$ is strongly elliptic. By assumption, $A+\diag(A)$ is invertible whenever $A\in L$ is nonzero, hence for $\xi\neq0$ and
$A\in L-\{0\}$,  we have
\begin{equation*}
    \langle A,\sigma( \Pi_L \sym\nabla (\sym\nabla)^*\Pi_L^* )(\xi)A\rangle_{\Rsn}=
    \left|\sigma((\sym\nabla)^*)(\xi)A\right|^2=
    \frac14\left|(A+\diag(A))\xi\right|^2 >0.
\end{equation*}
 Consider the function $\frac{\left|(A+\diag(A))\xi\right|^2}{|A|^2|\xi|^2}$ on $(L\setminus\{0\})\times(\mathbb R^n\setminus\{0\})$. It is  invariant under rescaling of both $A$ and $\xi$, therefore it is enough to restrict to the compact set $ \{(A,\xi)\in L\times\mathbb R^n:\ |A|=1,\ |\xi|=1\}.$  Since the function is strictly positive, it has a positive minimum. Thus there exists $\lambda>0$ such that $\frac14\left|(A+\diag(A))\xi\right|^2\geq\lambda |A|^2|\xi|^2$, which is exactly the \textbf{Legendre-Hadamard condition} as in Definition \ref{elliptic conditions}. By Theorem \ref{elliptic theorem}, the solvability of \eqref{equ-elliptic system for u_i} is established. For $\check\Phi=(\sym\nabla)^*\Pi_L^*u$, we obtain $\check\Phi\in C^{j+1,\alpha}(\overline\Omega,\mathbb R^n)$ and, by construction, $\Pi_L(D+\sym\nabla\check\Phi)=0$. This proves the lemma.
\end{proof}

\begin{rema}
    We remark that the formal adjoint has an explicit expression, $(\sym\nabla)^*A=-\frac{1}{2}\mathrm{div}(A+\diag (A))$, which can be obtained directly via integration by parts. For any compactly supported $v$, we have
    \begin{equation*}
\begin{aligned}
        \int_\Omega \langle \sym\nabla v, A \rangle_{\mathbb{R}^{n\times n}_{\mathrm{sym}}} \, dx 
        &= \int_\Omega \left( \sum_{i < j} \frac{1}{2}(\partial_i v_j + \partial_j v_i) a_{ij} + \sum_{i=1}^n (\partial_i v_i) a_{ii} \right) dx \\
        &= \frac{1}{2} \int_\Omega \sum_{i,j=1}^n (\partial_i v_j) (a_{ij} + \delta_{ij}a_{ii}) \, dx \\
        &= -\frac{1}{2} \int_\Omega \left\langle v, \mathrm{div}(A + \mathrm{diag}(A)) \right\rangle_{\mathbb{R}^n} \, dx.
    \end{aligned}
    \end{equation*}
    
\end{rema}

\begin{prop}\label{prop between two lemmas}
   The $\check{\Phi} $ constructed in Lemma \ref{elimination lemma} satisfies the estimate 
   
   \begin{equation}\label{estimates on check Phi}
       \lVert \check{\Phi}\rVert_{j+1,\alpha}\leq M_0 \lVert D\rVert_{j,\alpha},
   \end{equation}
   for some $M_0>0$ depending only on $j,\alpha,\Omega$, and $L$.
\end{prop}
\begin{proof}
    By \eqref{equ-elliptic system for u_i} and Theorem \ref{elliptic theorem}, we have
    \begin{equation*}
        \lVert \check{\Phi}\rVert_{j+1,\alpha}= \|(\sym\nabla)^*\Pi_L^*u\|_{j+1,\alpha}
        \lesssim  \|u\|_{j+2,\alpha}
        \lesssim \|\Pi_L D\|_{j,\alpha}\lesssim \|D\|_{j,\alpha}.\qedhere
\end{equation*}
\end{proof}

The following proposition shows that the requirement for $A + \mathrm{diag}(A)$ to be invertible for every non-zero matrix $A \in L$ is also necessary for the estimates of the form \eqref{estimates on check Phi}, in the sense that no uniform control could hold for all right-hand sides if the requirement fails.

\begin{prop}\label{prop-oscillatory construction for D}
    Let $L \subset \mathbb{R}^{n \times n}_{\mathrm{\mathrm{sym}}}$ be a linear subspace. Suppose there exists a non-zero matrix $A \in L$ such that $\det(A+\diag(A))=0.$
    Then, for any $M>0$, there exists
    $D\in C^{j,\alpha}(\overline{\Omega},\mathbb R^{n\times n}_{\mathrm{\mathrm{sym}}})$
    such that either the equation
    \begin{equation*}
        \Pi_L(\sym\nabla\check{\Phi}+D)=0
    \end{equation*}
    has no solution, or every solution
    $\check{\Phi}\in C^{j+1,\alpha}(\overline{\Omega},\mathbb R^n)$
    satisfies $\|\check{\Phi}\|_{j+1,\alpha}>M\|D\|_{j,\alpha}$.
    \end{prop}
\begin{proof}
Let $A \in L$ be a non-zero matrix  such that $\det(A+\diag(A))=0$. Then there exists a non-zero frequency vector $\xi^{(0)}\in\mathbb R^n$ such that
\begin{equation}\label{equ-degeneracy of symbol}
    (A+\diag(A))\xi^{(0)}=0.
\end{equation}
Recall that $\sigma((\mathrm{\mathrm{sym}}\nabla)^*)(\xi)A = -\frac{i}{2}(A + \mathrm{diag}(A))\xi$. Therefore, \eqref{equ-degeneracy of symbol} precisely yields 

\begin{equation}\label{equ-sigma degeneracy of symbol}
    \sigma((\Pi_L\sym\nabla)^*)(\xi^{(0)})A=\sigma((\sym\nabla)^*)(\xi^{(0)}) A=0.
\end{equation}

 Choose $\phi\in C_c^\infty(\Omega)$, $\phi\not\equiv 0$, and define
$ D_\lambda(x)=\phi(x)\sin(\lambda\xi^{(0)}\cdot x)A.$
Since $A\in L$, we have $\Pi_LD_\lambda=D_\lambda$. Moreover,
$  \|D_\lambda\|_{C^{j,\alpha}}\lesssim \lambda^{j+\alpha}$.\par 

We claim that if $\check{\Phi}_\lambda$ solves
$  \Pi_L(\sym\nabla\check{\Phi}_\lambda+D_\lambda)=0,$
then $\|\check{\Phi}_\lambda\|_{C^{j+1,\alpha}}
    \gtrsim \lambda^{j+1}.$ Indeed, let
$    e=\frac{\xi^{(0)}}{|\xi^{(0)}|},$
since $\Pi_L\sym\nabla$ has constant coefficients, in the sense of distributions,
$\Pi_L\sym\nabla(\partial_e^{j+1}\check{\Phi}_\lambda)=-\partial_e^{j+1}D_\lambda $.
Since the expansion $\partial_e^{j+1}D_\lambda=    \lambda^{j+1}|\xi^{(0)}|^{j+1}\phi(x)\sin(\lambda\xi^{(0)}\cdot x+\theta)A  +O(\lambda^j)$
in $C^0(\overline{\Omega})$, for the phase $\theta=(j+1)\pi/2$, by choosing the test function
$  \Psi_\lambda(x)=   \phi(x)\sin(\lambda\xi^{(0)}\cdot x+\theta)A ,$
then for $\lambda$ large we have
$$\int_\Omega
    \langle \partial_e^{j+1}D_\lambda,\Psi_\lambda\rangle\,dx
    \gtrsim \lambda^{j+1}.$$\par
On the other hand, 
denoting $(\Pi_L \sym\nabla)^*= \sum_{k=1}^n B_k \partial_k $ for some constant linear maps $B_k:\mathbb{R}^{n\times n}_{\mathrm{sym}}\rightarrow \mathbb{R}^n$, we then have $(\Pi_L \sym\nabla)^*\Psi_\lambda=\lambda \phi(x)\sin(\lambda \xi^{(0)}\cdot x+\theta)\sum_{i=1}^n\xi^{(0)}_kB_kA+O(1)=O(1)$, where $\xi_k^{(0)}$ is the $k$-th component of $\xi^{(0)}$, and we used \eqref{equ-sigma degeneracy of symbol}. Hence $   \|(\Pi_L\sym\nabla)^*\Psi_\lambda\|_{L^1}\lesssim 1$ uniformly in $\lambda$, which gives

\begin{equation*}
\begin{aligned}
    \int_\Omega
    \langle \partial_e^{j+1}D_\lambda,\Psi_\lambda\rangle\,dx &=
    -\int_\Omega    \langle \Pi_L\sym\nabla(\partial_e^{j+1}\check{\Phi}_\lambda),\Psi_\lambda\rangle\,dx  \\
    &=    -\int_\Omega    \langle \partial_e^{j+1}\check{\Phi}_\lambda,(\Pi_L\sym\nabla)^*\Psi_\lambda\rangle\,dx \lesssim\|\partial_e^{j+1}\check{\Phi}_\lambda\|_{C^0}.
\end{aligned}
\end{equation*} 
Therefore $ \lambda^{j+1}\lesssim 
\|\partial_e^{j+1}\check{\Phi}_\lambda\|_{C^0}\leq 
\|\check{\Phi}_\lambda\|_{C^{j+1,\alpha}}   $.

Combining the claim with
$   \|D_\lambda\|_{C^{j,\alpha}}\lesssim \lambda^{j+\alpha},$
we obtain
$    \frac{\|\check{\Phi}_\lambda\|_{C^{j+1,\alpha}}}{\|D_\lambda\|_{C^{j,\alpha}}}\gtrsim \lambda^{1-\alpha}$, which tends to infinity as $\lambda\to\infty$. Therefore, for any $M>0$, choosing $\lambda$ sufficiently large gives
$
    \|\check{\Phi}_\lambda\|_{C^{j+1,\alpha}}
    >
    M\|D_\lambda\|_{C^{j,\alpha}},
$ for this $D_\lambda$, either no solution exists, or every solution satisfies the claimed lower bound.
\end{proof}
\begin{lemm}[nonnegative coefficient lemma]\label{nonnegative coefficient lemma}
    Assume that $U\subset\mathbb{R}^{n\times n}_{\mathrm{sym}}$ is
   an $n_U$-dimensional subspace with a fixed basis
$\xi_1\otimes\xi_1,\ldots,\xi_{n_U}\otimes\xi_{n_U}$.  Then there exist  constant $M_1>0$ such that  for every  $\hat{D}\in C^{j,\alpha}(\overline{\Omega},U)$ there exists a linear function $\hat{\Phi}\in C^{\infty}(\overline{\Omega},\mathbb{R}^n)$ and functions $a_i \in C^{j,\alpha}(\overline{\Omega},\mathbb{R})$ satisfying
    \begin{equation}\label{coeffient equation}
\hat{D}+\sym\nabla\hat{\Phi}=\sum_{i=1}^{n_U} a_i^2 \xi_i\otimes\xi_i \ .
    \end{equation}
Moreover, $\nabla \hat{\Phi}$ is a constant matrix on $\Omega$, and  
\begin{equation}\label{nonnegative coefficients estimates}
\begin{aligned}
    \lVert\sum_{i=1}^{n_U} a_i^2 \xi_i\otimes\xi_i\rVert_{j,\alpha} +|\nabla\hat{\Phi}|+\|\hat{\Phi}\|_\alpha\leq M_1\lVert \hat{D}\lVert_{j,\alpha}.
\end{aligned}
\end{equation}
\end{lemm}
\begin{proof}
Write $\hat{D}=\sum_{i=1}^{n_U} \hat{a}_i \xi_i\otimes \xi_i$ with $\hat{a}_i\in C^{j,\alpha}(\overline{\Omega},\mathbb{R})$ that are not necessarily nonnegative. Since $\{\xi_i\otimes\xi_i\}$ is a fixed  basis of $U$,  there exists a constant $C(\xi_i)>0$, depending only on this basis, such that
    $\|\hat{a}_i\|_0\leq C(\xi_i)\|\hat{D}\|_0$ for any $i\leq n_U$. Denote  $\sigma_0:=\max_{i}\|\hat{a}_i\|_0$.  If $\sigma_0=0$, the statement is trivial. Assume now $\sigma_0>0$, we have 
\begin{equation*}
    a_i^2:=\hat{a}_i+2\sigma_0>0\quad \text{for some }\  a_i\in C^{j,\alpha}(\overline{\Omega},\mathbb{R}). 
\end{equation*}
For $x\in \overline{\Omega}$, letting
    \begin{equation}
        \hat{\Phi}=[\sum_{j=1}^n(\sum_{i=1}^{n_U}2\sigma_{0}\xi_i\otimes\xi_i)_{1j}x_j,\cdots,\sum_{j=1}^n(\sum_{i=1}^{n_U}2\sigma_{0}\xi_i\otimes\xi_i)_{nj}x_j]^T=2\sigma_0(\sum_{i=1}^{n_U}\xi_i\otimes\xi_i)x, 
    \end{equation}
    we obtain the linear function $\hat{\Phi}\in C^\infty(\overline{\Omega},\mathbb{R}^n)$ satisfying \eqref{coeffient equation}, with $\nabla \hat{\Phi}$  a constant matrix.\par

    Moreover,  there exists $C(n,\Omega,\alpha,\xi_i)>0$ depending on $n,\Omega,\alpha$ and the basis $\{\xi_i\otimes\xi_i\}$ such that
    \begin{equation}
|\nabla\hat{\Phi}|+\|\hat{\Phi}\|_\alpha=2\sigma_{0}|\sum_{i=1}^{n_U}\xi_i\otimes\xi_i |(1+n\|x\|_0)+[\hat{\Phi}]_\alpha\leq C(n,\Omega,\alpha,\xi_i)\|\hat{D}\|_\alpha,  
    \end{equation}
        \begin{equation}
 \lVert\sum_{i=1}^{n_U} a_i^2 \xi_i\otimes\xi_i\rVert_{j,\alpha}\leq\|\hat{D}\|_{j,\alpha}+|\nabla\hat{\Phi}|\leq (C(n,\Omega,\alpha,\xi_i)+1)\|\hat{D}\|_{j,\alpha}.
    \end{equation}
    Thus the estimate in \eqref{nonnegative coefficients estimates} follows.
\end{proof}
Lemma \ref{elimination lemma} removes a certain component of $D$, provided $L\subset\Rsn$ satisfies the  invertibility condition, by adding a controlled $\sym\nabla\check{\Phi}$. The remaining part lies in $U=L^\perp$, and serves as input to Lemma \ref{nonnegative coefficient lemma}. Hence it remains to construct such  $L$ with $L^\perp$ spanned by primitive matrices and with $\dim L$ maximal. This is done in Section \ref{section of intersections}, where we also  complete the proof of Lemma \ref{Decomposition Lemma}.

\subsection{Algebraic description}\label{subsection-algebraic description}
A feature of our approach is to reduce the analytic problem of improving the regularity in \eqref{ddagger} and \eqref{NK} to an algebraic optimization problem, via a decomposition method based on elliptic systems.  More precisely,  we formulate the algebraic problem as follows.

 \begin{prob}\label{affine problem}
Let $L \subset \mathbb{R}_{\mathrm{sym}}^{n\times n}$ be a subspace of real symmetric matrices with $\dim L = n_L$. Suppose $L$ satisfies the following two properties:
\begin{enumerate}[label=\textcolor{blue}{\arabic*.}, ref=\arabic*]
\item For every non-zero matrix $A \in L$, the matrix $A + \operatorname{diag}(A)$ is invertible.\label{condition 1 in affine problem}
    \item The subspace $L^\perp$ is spanned by rank-one symmetric matrices $\xi_1\otimes\xi_1, \dots, \xi_m\otimes\xi_m \in\mathbb{R}^{n\times n}_{\mathrm{sym}}$, where $m=\frac{n(n+1)}{2}-n_L$.\label{condition 2 in affine problem}
\end{enumerate}

 For each $n$, find an $L$ that maximizes $n_L$.
 \end{prob} Recall that the condition \ref{condition 1 in affine problem} is needed in the elimination lemma \ref{elimination lemma} to ensure that the constructed system \eqref{equ-elliptic system for u_i} is elliptic, and the second condition is needed in the nonnegative coefficient lemma \ref{nonnegative coefficient lemma} to ensure $\hat{D}\in L^\perp$.

 It is natural to consider Problem \ref{affine problem} under projectification and in the language of intersection. Denote 
\begin{equation}\label{defi of Y}
    \bY:=\{
    [y_{ij}]\in\PRs\ \big| \mathrm{ det}\left((y_{ij})+\mathrm{diag}(y_{ij}) \right)=0\}\subset \boldsymbol{\rm P}(\mathbb{R}_{\mathrm{sym}}^{n\times n}),
\end{equation}
thus $\bY$ is a degree $n$ hypersurface in $\boldsymbol{\rm P}(\mathbb{R}_{\mathrm{sym}}^{n\times n})\cong\mathbb{R}\boldsymbol{\rm P}^{\frac{n(n+1)}{2}-1}$. Additionally, denote 
\begin{equation}\label{defi of Z}
    \bZ:=\{[w_1w_2,\cdots,w_{n-1}w_n,w_1^2,\cdots, w_n^2]\ \big| [w_1,\cdots,w_n]\in \mathbb{R}\boldsymbol{\rm P}^{n-1}\}\subset \boldsymbol{\rm P}(\mathbb{R}_{\mathrm{sym}}^{n\times n})^\vee,
\end{equation}
namely $\bZ$ is the image of $\mathbb{R} \bP^{n-1}$ in $\boldsymbol{\rm P}(\mathbb{R}_{\mathrm{sym}}^{n\times n})^\vee$ via \textbf{Veronese embedding} $\iota: \mathbb{R}\bP^{n-1}\rightarrow \boldsymbol{\rm P}(\mathbb{R}_{\mathrm{sym}}^{n\times n})^\vee$. We denote $\bL=\bP(L)$. Using the inner product $\langle \,,\,\rangle_{\mathbb{R}^{n\times n}_{\mathrm{sym}}}$ to identify
$\mathbb P(\mathbb R^{n\times n}_{\mathrm{\mathrm{sym}}})$ with its dual projective space.  Under this identification, the projective dual of $\bL$ is given by
\begin{equation*}
    \bL^\vee=\mathbf P(L^\perp),
\end{equation*}
 with the precise definition of projective duality  deferred to Section \ref{sec-projective duality}. Recall that
 a variety $X\subset \RPn$ is said to be \textbf{nondegenerate} if it is not contained in any hyperplane of $\RPn$, and it is equivalent to the existence of $n+1$ points in $X$ that span $\RPn$. Thus the following gives an equivalent formulation of Problem \ref{affine problem}, with the dimension shift $n_{\bL}=n_L-1$ due to projectivization.
 
\begin{prob}\label{proj problem}
    Consider an $n_{\bL}$ dimensional linear subspace $\bL$ of $\bP(\Rsn)$, and its dual space $\bL^\vee$.  Let $\bL$ satisfy the following conditions on $\bY\subset \bP(\mathbb{R}^{n\times n}_{\mathrm{sym}})$ and $\bZ\subset \bP(\mathbb{R}^{n\times n}_{\mathrm{sym}})^\vee$:
    \begin{enumerate}
        \item $\bL\cap \bY=\varnothing$, \quad and
        \item  $\bL^\vee \cap \bZ$  is nondegenerate in $\bL^\vee$. 
    \end{enumerate}
    For each $n$, find an $\bL$ that maximizes $n_{\bL}$.
\end{prob}
\begin{rema}
    The fact of considering the intersection Problem \ref{proj problem}  over the field $\mathbb{R}$ instead of $\mathbb{C}$ makes it interesting. Over the field $\mathbb{C}$, since  $\bY$ is a hypersurface of $\mathbb{C}\bP^{\frac{n(n+1)}{2}-1}$, by classical intersection theory we have $\bL\cap \bY\neq \varnothing$ for any linear subspace $\bL$ with $\dim_{\mathbb{C}}\bL> 0$. 
\end{rema}

\section{The intersections in real projective spaces}\label{section of intersections}
\subsection{Projective duality}\label{sec-projective duality}
Projective duality is a basic concept in algebraic geometry that describes a symmetry between a projective variety and the locus of its tangent hyperplanes. This elegant framework is valid over both the real and complex fields and provides a systematic way to recover a variety from its dual. For background, we refer the reader to \cite{harris_algebraic_1992, tevelev_rojective_2004}; the necessary notions are briefly recalled below.\par
Let $X \subset \mathbb{R}\bP^m$ be a projective variety. A hyperplane $H \subset \mathbb{R}\bP^m$ is called a \textbf{tangent hyperplane} if it contains the tangent space to $X$ at some smooth point. Since each hyperplane in $\mathbb{R}\bP^m$ corresponds to a unique point in the dual space $(\mathbb{R}\bP^m)^\vee$, we define the \textbf{dual variety} $X^\vee \subset (\mathbb{R}\bP^m)^\vee$ as the closure of the locus of all tangent hyperplanes to $X$ at its smooth points. We recall the following fundamental results from projective duality.\par
\begin{theo}[Reflexivity Theorem]\quad
        $(X^\vee)^\vee=X.$
\end{theo}
\begin{prop}[Principle of duality]\label{principle of duality}
Let $\mathrm{pt}_1$, $\mathrm{pt}_2$ be two points and  $\bL$ be a linear subspace of $\RP^m$. If $\mathrm{pt}_1,\mathrm{pt}_2\in \bL\subset\RP^m$, then 
\begin{equation*}
    \bL^\vee\subset(\mathrm{pt}_1^\vee\cap \mathrm{pt}_2^\vee)\subset (\RP^m)^\vee.
\end{equation*}

\end{prop}

Recall  $\bY\subset \bP(\mathbb{R}_{\mathrm{sym}}^{n \times n})$ defined in \eqref{defi of Y}, $\bZ\subset \bP(\mathbb{R}_{\mathrm{sym}}^{n \times n})^\vee$ defined  in \eqref{defi of Z}. Notice that $\bZ$ is a \textit{determinantal variety of rank one} in $\bP(\mathbb{R}_{\mathrm{sym}}^{n \times n})^\vee$. Its singular locus is contained in the determinantal variety of rank $0$, which is $\varnothing \in \bP(\mathbb{R}_{\mathrm{sym}}^{n \times n})^\vee$, hence $\bZ$ is smooth. By the classical theory of quadratic forms, we can also view $\bZ\subset\bP(\mathbb{R}_{\mathrm{sym}}^{n \times n})^\vee$ as the set of quadratic forms, and $\bY\subset \PRs$ as the \textit{discriminant variety $\Delta$} of $\bZ$. This yields the following.
\begin{prop}$\bY^\vee=\bZ$. 
\end{prop}

The following transversality result for $\bZ\cap \bL^\vee$ is an immediate consequence of projective duality. Recall that two smooth subvarieties $X,Y\subset \mathbb{R}\bP^m$ are said to \textbf{intersect transversely} if $T_pX+ T_pY=T_p\mathbb{R}\bP^m$ for every $p\in X\cap Y$.

\begin{prop}\label{prop-Z intersect L^vee transversally}
    Let $\bL$ be a linear subspace of $\bP(\Rsn)$. If $\bL\cap \bY=\varnothing$ and $\bL^\vee \cap \bZ\neq \varnothing$, then $\bZ $ intersects $\bL^\vee$ transversally.
\end{prop}
\begin{proof}
    Assuming that the intersection is not transverse at some point
$z\in \bL^\vee\cap \bZ$, then there would exist a hyperplane $H\subset \PRs^\vee$
containing both $\bL^\vee$ and the tangent space $T_z\bZ$. Thus, $H$ is a tangent hyperplane to $\bZ$ and contains $\bL^\vee$, so its dual point $H^\vee$ lies in $\bY\cap \bL$,
contradicting $\bY\cap \bL=\varnothing$.
\end{proof}

We observe a duality between $\bY$ and $\bZ$, which provides a geometric connection between the two conditions in Problem \ref{proj problem}. These conditions are naturally encoded in the decomposition lemma.

\begin{exam}[$n=2$ case]\normalfont\label{CS example}

As explained in \cite{lewicka_convex_2017,cao_very_2019}, the problem \eqref{ddagger} in  dimension $2$ is particularly important since it directly relates with the 2-dimensional Monge--Amp\`ere equation. This example illustrates how our method recovers the decomposition established in \cite[Proposition 3.1]{cao_very_2019}, achieving regularity $C^{1,\alpha}$ for $\alpha<\frac{1}{5}$, while further demonstrating that this decomposition is not unique.  \par
    
    For $n=2$, $\bY=\{4y_2y_3-y_1^2=0\}\subset \mathbb{R}\bP^2$, $\bZ=\{[w_1 w_2:w_{1}^2:w_{2}^2]\ \big|\  [w_1:w_2]\in\mathbb{R}\bP^1\}=\{z_{2}z_{3}-z_{1}^2=0\}\subset (\mathbb{R}\bP^2)^\vee$.  
    We aim to determine the maximum dimensional linear subspace $\bL\subset \mathbb{R}\bP^2$  such that $\bL\cap \bY=\varnothing$ and  $|\bL^\vee \cap \bZ|=\dim \bL^\vee+1$.  \par

\begin{figure}[h!]
\centering
\begin{subfigure}[t]{0.49\textwidth}
\centering
\begin{tikzpicture}[scale=0.5]
\def\a{1}
\def\L{1} % width of interval

\pgfmathsetmacro{\Va}{8/\a+1} \pgfmathresult

\draw[-,thick] (-4.5,0) -- (4.5,0) coordinate (x axis) node[below] {$\frac{y_2-y_3}{y_2+y_3}$};
\draw[-,thick] (0,-3.5)  -- (0,4) coordinate (y axis) node[left] {$\frac{y_1}{y_2+y_3}$};
\draw (0,0) ellipse (2 and 2); 
\draw (-1.4,-1.2) node[right]{ $\bY$};

\filldraw (2,-2) circle (0.15);
\draw (2,-2.45) node[right]{$\bL_{0}$};

\draw[thick,dashed] (2,3)--(2,-3);
\draw (1.9,2) node[right]{$\rm (pt_2)^\vee$};

\draw[thick,dashed] (-3,-2)--(3,-2);
\draw (-2.5,-2.5) node[right]{$\rm (pt_1)^\vee$};

\draw[thick] (1,2.88)--(-5,-0.58);
\draw (-1.5,1.7) node[left]{$\bL_{1}'$};

\draw[thick] (1,4.38)--(-5,0.92);
\draw (-1.7,3) node[left]{$\bL_{1}$};

\draw (3,3.8) node[right]{in $\RP^2$};
\end{tikzpicture}
  \centering
 \\  $\bL_{1}\cap \bY=\varnothing$ and $\bL_1'$ tangent to $\bY$ as lines,\\
 $\bL_{0}$ is the intersection point of lines $\rm pt_1^\vee$, $\rm pt_2^\vee$. 
\end{subfigure}
\rulesep
\begin{subfigure}[t]{0.45\textwidth}
\centering
\begin{tikzpicture}[scale=0.5]
\def\a{1}
\def\L{1} % width of interval

\pgfmathsetmacro{\Va}{8/\a+1} \pgfmathresult

\draw[-,thick] (-4.5,0) -- (5,0) coordinate (x axis) node[below] {$\frac{z_2-z_3}{z_2+z_3}$};
\draw[-,thick] (0,-3)  -- (0,3.5) coordinate (y axis) node[left] {$\frac{z_1}{z_2+z_3}$};
\path (0,-3)--(0,-3.5);

\draw (0,0) ellipse (3 and 1.5); 
\draw (1.5,1.7) node[right]{ $\bZ$};

\draw[thick] (2,2.5)--(-4.5,-0.75);
\draw (-1.7,0.4) node[right]{$\bL_{0}^\vee$};

\filldraw (1.5,-1.3) circle (0.15);
\draw (1.5,-1.5) node[right]{$\bL_{1}'{}^\vee$}; 

\filldraw (0.7,-0.7) circle (0.15);
\draw (0.7,-0.7) node[right]{$\bL_{1}^\vee$};

\filldraw (0,1.5) circle (0.15);
\draw (0,1.75) node[left]{\small $\rm pt_1$};

\filldraw (-3,0) circle (0.15);
\draw (-2.8,0.5) node[left]{\small $\rm pt_2$};

\draw (3,3.8) node[right]{in $(\RP^2)^\vee$};
\end{tikzpicture}
\centering\\
   $\bL_{1}^\vee\notin \bZ$ and $\bL_{1}'{}^\vee\in \bZ$ as points,\\
    $\bL_{0}^\vee$ is the line crossing $\rm pt_1$, $\rm pt_2$.
\end{subfigure}
    \caption{Projective duality of $\bY$ and $\bZ$ when $n=2$.} 
    \label{fig:enter-label}
\end{figure}
    
    In fact, if $\dim \bL=1$, then $\bL^\vee$ is a point, hence $|\bL^\vee \cap \bZ|=\dim \bL^\vee+1$ is equivalent to $\bL^\vee \in Z$. Since $Y^\vee=Z$, by Proposition \ref{principle of duality}, $\bL^\vee\in Z$ if and only if $\bL$ tangent to  $\bY$, which obligates the condition $\bL\cap \bY=\varnothing$.\par
    If $\bL$ is a point in $\mathbb{R}\bP^2$, then $\bL^\vee$ is a line in $(\mathbb{R}\bP^2)^\vee$, thus $\bL^\vee$ intersects  $\bZ$ at two points $\mathrm{pt}_1,\mathrm{pt}_2$ if and only if  $\mathrm{pt}_1^\vee \cap \mathrm{pt}_2^\vee=\bL$. Thus $\dim (\bL)=0$ is the maximum dimension of $\bL$ that staisfies the second condition in Problem \ref{proj problem}. \par

    In Proposition 3.1 of \cite{cao_very_2019}, the authors particularly choose $\bL_{cs}=\{[y_1:y_2:-y_2]\}\subset\mathbb{R}\bP^2$, hence $\dim \bL_{cs}=1$, and $\bL_{cs}\cap \bY=\varnothing$. Consequently, as a point, $\bL_{cs}^\vee=[0:1:1]\notin \bZ$ and corresponds to matrix $Id$.  In \cite{cao_very_2019} the authors moreover choose two points $[0:1:0],\  [0:0:1]\in \bZ$ to generate the point $\bL_{cs}^\vee$.    ~\par

\begin{figure}[h!]\label{Lcs and L*}
\centering
\begin{subfigure}[t]{0.49\textwidth}
\centering
\begin{tikzpicture}[scale=0.5]
\def\a{1}
\def\L{1} % width of interval

\pgfmathsetmacro{\Va}{8/\a+1} \pgfmathresult

\draw[-,thick] (-4,0) -- (3.5,0) coordinate (x axis) node[below] {$\frac{y_2-y_3}{y_2+y_3}$};
\draw[-,thick] (0,-3.5)  -- (0,4) coordinate (y axis) node[left] {$\frac{y_1}{y_2+y_3}$};
\draw (0,0) ellipse (2 and 2); 
\draw (-1.4,-1.2) node[right]{ $\bY$};

\draw[blue,thick,dashed] (-5,-3.5)--(-5,4);
\draw[blue] (-5.1,2.5) node[left]{$\bL_{cs}$};
\draw[blue] (-5,1.8) node[left]{$=\bL_{\infty}$};

\draw (2.5,3.8) node[right]{in $\RP^2$};

\end{tikzpicture}
  \centering
 \\ $\bL_{cs}$  is  the line at infinity $[y_1:y_2:-y_2]$.
\end{subfigure}
\rulesep
\begin{subfigure}[t]{0.45\textwidth}
\centering
\begin{tikzpicture}[scale=0.5]
\def\a{1}
\def\L{1} % width of interval

\pgfmathsetmacro{\Va}{8/\a+1} \pgfmathresult

\draw[-,thick] (-4.5,0) -- (5,0) coordinate (x axis) node[below] {$\frac{z_2-z_3}{z_2+z_3}$};
\draw[-,thick] (0,-3)  -- (0,3.5) coordinate (y axis) node[left] {$\frac{z_1}{z_2+z_3}$};
\path (0,-3)--(0,-3.5);

\draw (0,0) ellipse (3 and 1.5); 
\draw (1.5,1.7) node[right]{ $\bZ$};

\filldraw[blue] (0,0) circle (0.15);
\draw[blue] (0.25,0.55) node[left]{$\bL_{cs}^\vee$};

\filldraw (3,0) circle (0.15);
\draw (2.8,0.5) node[right]{\small $[0:1:0]$};

\filldraw (-3,0) circle (0.15);
\draw (-2.8,0.5) node[left]{\small $[0:0:1]$};

\draw (3,3.8) node[right]{in $(\RP^2)^\vee$};
\end{tikzpicture}
\centering\\
  $\bL_{cs}^\vee$ is the origin $[0:1:1]$.
\end{subfigure}
    \caption{The choice of  $\bL_{cs}$ in Cao-Sz\'ekelyhidi \cite{cao_very_2019}} 
    \label{fig: CS}
\end{figure}
    
\end{exam}

The preceding example reveals a fundamental limitation for $n=2$ (a phenomenon also occurring for $n \in \{4, 8, 16\}$, as detailed in Section \ref{section: nondeg of Z}): 
\begin{equation*}
\text{any maximal linear subspace $\bL$ disjoint from $\bY$ necessarily satisfies }\bL^\vee\cap\bZ=\varnothing.
\end{equation*}
This geometric obstruction explains why, in \cite{cao_very_2019}, the authors reduced the matrix to the single-parameter form $d^2Id$, yet still required two rank-one symmetric matrices to span it. Their diagonalization proposition, recovered naturally via the elliptic system in our Lemma \ref{elimination lemma}, is stated as follows.

\begin{prop}[Cao-Sz\'ekelyhidi {\cite[Proposition 3.1]{cao_very_2019}}]
For $n=2$, there exist constants $M_1,M_2,\sigma_1>0$ depending only on $j,\alpha$, such that for every $D\in C^{j,\alpha} (\overline{\Omega},\mathbb{R}^{2\times2}_{\mathrm{sym}})$ with $\|D-Id\|_\alpha\leq \sigma_1$, by letting $\phi,\psi\in C^{j+2,\alpha}(\overline{\Om},\mathbb{R})$ satisfying
\begin{equation*}
    \begin{cases}
    2\Delta\phi=\frac{1}{\sqrt{2}}(D_{11}-D_{22}),\quad \Delta\psi=D_{12} & \text{in }\Omega\\
        \phi=\psi=0& \text{on } \partial \Omega
    \end{cases} 
\end{equation*}
and letting
\begin{equation*}
\Phi_{cs}:=-2\,[\sqrt{2}\partial_1\phi+\partial_2\psi,\,-\sqrt{2}\partial_2\phi+\partial_1\psi]^T\in C^{j+1,\alpha}(\overline{\Omega},\mathbb{R}^2),     
\end{equation*}
one has 
\begin{equation}
    D+\sym \nabla \Phi_{cs}=\begin{bmatrix}
        D_{11}-2\sqrt{2}\partial_1^2\phi-2\partial_1\partial_2\psi&0\\
        0&D_{22}+2\sqrt{2}\partial_2^2\phi-2\partial_2\partial_2\psi
    \end{bmatrix}=d^2 Id,
\end{equation}
for some $d\in C^{j,\alpha}(\overline{\Omega},\mathbb{R})$,
with the following estimates:
\begin{gather}
    \|d-1\|_\alpha+\|\nabla\Phi_{cs}\|_\alpha\leq M_1\|D-Id\|_\alpha,\\
[d]_{j,\alpha}+[\nabla \Phi_{cs}]_{j,\alpha}\leq M_2 \|D-Id\|_{j,\alpha}.
\end{gather}
    
\end{prop}

\subsection{Radon--Hurwitz number: symmetric matrices case}\label{R-H symmetric cases}

As discussed in the Preliminaries, the Radon--Hurwitz number $\rho(n)$ provides the sharp upper bound for the maximal dimension of a vector space $W$ of real $n \times n$ matrices in which every nonzero element is invertible, as established by Adams in \cite{Adams_1962_vectorfields} via topological $K$-theoretic methods.\par

In this paper, we focus on the symmetric case. The condition \ref{condition 1 in affine problem} in Problem \ref{affine problem} is equivalent to finding invertible subspaces of symmetric matrices in the sense of Section \ref{RH number: a review}. This equivalence is formalized  in the following elementary proposition, whose proof is omitted.\par

\begin{prop}\label{prop-diagonal doubled}
  For any $L\subset \mathbb{R}_{\mathrm{sym}}^{n \times n}$,  let $L_2$ be the linear subspace of $\mathbb{R}_{\mathrm{sym}}^{n\times n}$ given by $L_2:=\{A+\diag(A)\ \big|\   A\in L\}$, namely $L_2$ is the space of matrices from $L$ with their diagonal doubled. Then $\dim L_2=\dim L$, and $L_2$ is invertible  if and only if $\bP(L)\cap \bY=\varnothing$.
\end{prop}

In \cite{ALP_1965,ALP_1966_correction}, Adams, Lax, and Phillips determined the maximum dimension of invertible spaces of symmetric real matrices using the 8-fold periodicity of the Radon--Hurwitz number and an elementary construction as in Proposition \ref{prop-construction of W_n}, while the upper bound relies on Adams's theorem \ref{adams theorem} from \cite{Adams_1962_vectorfields}.

\begin{theo}[Adams-Lax-Phillips \cite{ALP_1965}] \label{Thm ALP}
The maximum dimension of an invertible space of $n \times n$ symmetric real matrices $W_{n,\mathrm{sym}}\subset\mathbb{R}_{\mathrm{sym}}^{n \times n}$ is $\rho(\frac{1}{2}n)+1$, where $\rho(n)$ is the Radon--Hurwitz number, and $\rho(\frac{1}{2}n)$ is set to be $0$ if $\frac{1}{2}n$ is not an integer.
    
\end{theo}
Recall that the \textbf{signature} of a real symmetric matrix is the pair $(a,b)$, where $a$ (resp.$b$) is the number of positive (resp. negative) eigenvalues, counted with multiplicities. The Continuity of eigenvalues yields the following.\par

\begin{prop}\label{prop-signature is constant}
Let $W\subset \mathbb R^{n\times n}_{\mathrm{sym}}$ be an invertible linear subspace. Then the signature is constant on each connected component of $W-\{0\}$. In particular, if $\dim W\geq 2$, then all nonzero elements of $W$ have the same signature.
\end{prop}

\begin{proof}
Let $A(t)$, $t\in[0,1]$, be a continuous path in $W-\{0\}$. Since $W$ is invertible, each $A(t)$ is invertible as a matrix. The eigenvalues of a real symmetric matrix depend continuously on the matrix entries. Hence, along the path $A(t)$, no eigenvalue can cross zero. Therefore the number of positive eigenvalues and the number of negative eigenvalues are constant along the path. It follows that the signature is constant on each connected component of $W-\{0\}$. If $\dim W\geq 2$, then $W-\{0\}$ is connected. Hence the signature is constant on all of $W-\{0\}$.
\end{proof}

As an immediate consequence, we recover the odd-dimensional case of Theorem~\ref{Thm ALP}.
\begin{coro}\label{coro-odd dim case}
    Maximum  dimension of an invertible $W_{n,\mathrm{sym}}$ is $1$ when $n$ is odd.
\end{coro}
\begin{proof}
The existence of a one-dimensional invertible subspace is immediate. Assume, for the contradiction, that $W\subset \mathbb R^{n\times n}_{\mathrm{\mathrm{sym}}}$ is
invertible and $\dim W\geq 2$. Then by proposition~\ref{prop-signature is constant}, each element $A\in W_{n,\mathrm{sym}}$ has the same signature, say $(a,n-a)$. However, $-A$ also lies in $W_{n,\mathrm{\mathrm{sym}}}$, and its signature is $(n-a,a)$. Hence $n-a=a$, contradicting the assumption that $n$ is odd.
\end{proof}
Note that the same conclusion extends to the not necessarily symmetric case. Although the signature is not defined for general real matrices, one may instead use the continuity of the determinant along a path in $W-\{0\}$ from $A$ to $-A$ to obtain the conclusion.\par

\begin{coro}\label{coro-even-dim-signature}
    When $n$ is even and $\dim W_{n,\mathrm{sym}}\geq 2$, each matrix in invertible $W_{n,\mathrm{sym}}$ is of signature $(\frac{n}{2},\frac{n}{2})$.
\end{coro}

\begin{proof}
By Proposition~\ref{prop-signature is constant}, all nonzero elements of $W$ have the
same signature, denote as $(a,n-a)$.
Since $-A\in W-\{0\}$, its signature $(n-a,a)$ must be the same with that of $A$. Therefore $a=n-a$, and hence $a=n/2$.
\end{proof}

Using the constructions from Proposition \ref{prop-construction of W_n}, we now present the proof of Theorem \ref{Thm ALP} due to Adams, Lax, and Phillips \cite{ALP_1965}, for the sake of completeness.
\begin{proof}[proof of Theorem \ref{Thm ALP} due to \cite{ALP_1965}]

Let $W_n$ be an invertible space of $n\times n$ real matrices (not necessarily symmetric) of dimension $\rho(n)$, which is maximal due to \cite{Adams_1962_vectorfields}. Let $W_{2n,\mathrm{sym}}\subset \mathbb{R}^{2n\times 2n}_{\mathrm{sym}}$ be an invertible subspace; our goal is to determine the maximum of $\dim W_{2n,\mathrm{sym}}$.\par
    Consider the vector space of $2n \times 2n$ symmetric matrices of the form $\begin{bmatrix}
        rI_{n} &A \\
        A^T &  -rI_n
    \end{bmatrix}$ for each $r\in \mathbb{R}$ and $A\in W_n$. It is straightforward to verify that this space is invertible and of dimension $\rho(n)+1$, hence
   $\dim W_{2n,\mathrm{sym}}\geq \rho(n)+1.$  \par
   On the other hand,  for $\iota\in\operatorname{Im}\mathbb O$, let $\mathsf{L}_\iota:\mathbb O\to\mathbb O$ denote left multiplication by $\iota$, one regards $\mathsf{L}_\iota$ as an $8\times 8$ real skew-symmetric matrix. Given an invertible subspace
$V\subset \mathbb R_{\mathrm{\mathrm{sym}}}^{2n\times 2n}$, define
\begin{equation*}
W'_{16n}=\{A\otimes I_8+I_{2n}\otimes \mathsf{L}_\iota\mid A\in W_{2n,\mathrm{sym}},\ \iota\in \operatorname{Im}\mathbb O\}
\subset \mathbb R^{16n\times 16n}.
\end{equation*}
Then $ \dim  W_{2n,\mathrm{sym}}+7=\dim W'_{16n}\leq\rho(16n)$.    The $8$-fold periodicity in the definition of $\rho(n)$ ensures $\rho(16n)=\rho(n)+8$, thus $\rho(n)+1 \leq \dim W_{2n,\mathrm{sym}}\leq \rho(n)+1$. The odd-dimensional case was proved in Corollary \ref{coro-odd dim case}, and hence the theorem follows for all $n$.
\end{proof}

\subsection{Nondegeneracy of  \textbf{Z} and  intersection of  \textbf{Z} with  \textbf{L}${}^\vee$}\label{section: nondeg of Z}
In this subsection, we use classical algebraic geometry over $\mathbb{R}$ to prove the following theorem, which answers Problem \ref{proj problem}. The necessary ingredients will be developed throughout this subsection, and the proof will be given at the end.
\begin{prop}[Theorem \ref{second theorem}] \label{prop: answer for problem 2}
    \begin{equation*}
    \begin{array}{c}
          \text{The maximum dimension  of $L\subset\mathbb{R}_{\mathrm{sym}}^{n\times n}$} \\
        \text{  that satisfies the two conditions of Problem \ref{affine problem}} 
    \end{array}=\begin{cases}
    \rho(\frac{1}{2}n) &\text{ for } n= 2,4;\\
            \rho(\frac{1}{2}n)+1 \text{ or } \rho(\frac{1}{2}n)& \text{ for } n= 8,16;\\
           \rho(\frac{1}{2}n) +1 & \text{\rm for all other } n\in \mathbb{Z}_{\geq2}.\\
        \end{cases}
    \end{equation*} 
 \end{prop}

In this direction, we  construct a subspace $L\subset\Rsn$ with $\dim L=\frac{n(n+1)}{2}-\Xi_n$, attaining the lower value in Proposition \ref{prop: answer for problem 2}, satisfying $\bP(L)\cap \bY=\varnothing$, and $\bP(L)^\vee\cap \bZ\neq \varnothing$. Recall the iterative construction of $W_n$ as in Proposition \ref{prop-construction of W_n}.

\begin{prop}\label{prop-construction and nonempty intersection}
For the projective linear subspace
$\bL=\bP(L)$, where the invertible subspace $L\subset\mathbb{R}^{n\times n}_{\mathrm{sym}}$ is defined as follows, one has $\bZ\cap \bL^\vee\neq \varnothing$.\par

Moreover, for these choices of $L$, the diagonal-doubling map $A\mapsto A+\diag(A)$ preserves $L$ as a subspace: $ \{A+\diag(A)\mid A\in L\}=L $. Hence, for these $L$, the condition $\bP(L)\cap\bY=\varnothing$ is equivalent to the ordinary invertibility of every nonzero element of $L$.\par

\begin{enumerate}
    \item If $n=2,4,8,16$, then  $L=\left\{\begin{bmatrix}
            0&A\\
            A^T&0
        \end{bmatrix}\mid A\in W_{\frac{n}{2}} \right\}$. In this case, $[e_1\otimes e_1]\in \bZ\cap \bL^\vee$.
     \item  If $n$ is odd, then $L= \left\{\begin{bmatrix}
            \frac{r}{n-1}I_{n-1}&0\\
            0&-r 
        \end{bmatrix}\mid r\in \mathbb{R} \right\}$. In this case, $[(\sum_{i=1}^n e_i)\otimes (\sum_{i=1}^n e_i)]\in \bZ\cap \bL^\vee$.
    \item If $n$ is even, $n\neq 2,4,8,16$, then $L=\left\{\begin{bmatrix}
            rI_{\frac{n}{2}}&A\\
            A^T&-r I_{\frac{n}{2}}
        \end{bmatrix}\mid A\in W_{\frac{n}{2}}, \, r\in \mathbb{R} \right\}$. In this case,
        with the explicit construction of $W_{\frac n2}$, one has
          $[(e_1+e_{24})\otimes (e_1+e_{24})]\in \bZ \cap \bL^\vee$ when $n=32$, while $[(e_1+e_{n})\otimes (e_1+e_{n})]\in \bZ \cap \bL^\vee$ in the remaining even cases.
\end{enumerate}

\end{prop}
\begin{proof}
The statement on double-diagonal is straightforward. For Cases 1 and 2, it is immediate that $L$ is invertible and that the indicated
rank-one points belong to $\bZ\cap\bL^\vee$.\par
For case 3, following the proof of Theorem \ref{Thm ALP}, the $L$ is invertible. When $n=32$, recall the construction of $W_{16}$, the matrices in $L$ have the form 
\begin{equation*}
    \begin{bmatrix}
        r_1I_{16}& \begin{matrix}
            r_2 I_8& \mathsf{L}_{\mathbb{O}}\\
            \mathsf{L}_{\mathbb{O}}^T& -r_2I_8
        \end{matrix}
            \\
     \begin{matrix}
            r_2 I_8& \mathsf{L}_{\mathbb{O}}^T\\
            \mathsf{L}_{\mathbb{O}}& -r_2I_8
        \end{matrix}   & -r_1I_{16}
    \end{bmatrix},\quad         \text{with }r_1,r_2\in \mathbb{R}, 
\end{equation*}
where        $\mathsf{L}_{\mathbb{O}}$ ranges over the real matrices representing  left multiplication of $\mathbb{O}$. 
In particular, the $(1,24)$-entry of every matrix in $L$ vanishes, hence $[(e_1+e_{24})\otimes (e_1+e_{24})]\in \bZ \cap \bL^\vee$ when $n=32$. For the remaining even-dimensional cases, $[(e_1+e_n)\otimes (e_1+e_n)]\in \bZ\cap \bL^\vee$ follows easily from  the inductive construction of $W_n$ in Proposition \ref{prop-construction of W_n}.
\end{proof}

The next proposition shows that the above $L$ can be perturbed to a  $L'$ of the same dimension $\frac{n(n+1)}{2}-\Xi_n$, such that $\bZ\cap \bP(L')^\vee$ is nonempty and nondegenerate in $\bP(L')^\vee$. Recall that a variety is
\textbf{irreducible} if it cannot be written as the union of two proper closed
subvarieties, and that a subset is \textbf{Zariski dense} if every polynomial vanishing on
the subset vanishes on the whole variety. 

\begin{prop}\label{prop-nondegenerancy}
Let $L\subset \Rsn$ be as in Proposition \ref{prop-construction and nonempty intersection}, and assume $n\geq 3$. Then, for
every sufficiently small neighborhood of $L$ in Grassmannian
$ \mathrm{Gr}(\dim L,\Rsn)$,
there exists a linear subspace   
\begin{equation*}
L'\subset \Rsn,
    \qquad
    \dim L'=\dim L,
    \end{equation*}
in this neighborhood such that  $\bY\cap \bP(L')=\varnothing$ and the real intersection $\bZ\cap \bP(L')^\vee$ is nondegenerate in $\bP(L')^\vee$.
\end{prop}

\begin{proof}
Notice that by the construction of $L$ in Proposition \ref{prop-construction and nonempty intersection}, $L$ is invertible and $\bY\cap \bP(L)=\varnothing$.
    Suppose $[\mathbf{x}_0\otimes \mathbf{x}_0]\in \bZ\cap \bP(L)^\vee$, namely $\mathbf{x}_0^T A \mathbf{x}_0=0$ for all $A\in L$.
We first show that the real point persists under small perturbations. Choose a
basis $A_1,\ldots,A_{\dim L}$ of $L$. For $\tilde{L}$ sufficiently close to $L$ in the
Grassmannian, choose a basis $A_1(\tilde{L}),\ldots,A_{\dim L}(\tilde{L})$ of $\tilde{L}$ depending
smoothly on $\tilde{L}$, with $A_i(L)=A_i$. Define
\begin{equation*}
    \mathcal{F}:\mathbb{R}^n-\{0\}\times \mathrm{Gr}(\dim L,\Rsn)\longrightarrow \mathbb{R}^{\dim L},\qquad
    \mathcal{F} (x,\tilde{L})=\left(x^TA_1(\tilde{L})x,\cdots,x^TA_{\dim L}(\tilde{L})x\right).
\end{equation*}
Then $\mathcal F(\mathbf{x}_0,L)=0$, and its differential with respect to the variable $x$ at point $(\mathbf{x}_0,L)$ is
\begin{equation*}
    D_\mathbf{x}\mathcal{F}_{(\mathbf{x}_0,L)}:\mathbb{R}^n\longrightarrow \mathbb{R}^{\dim L},\qquad
    D_\mathbf{x}\mathcal{F}_{(\mathbf{x}_0,L)}(\xi)=\left(2\xi^TA_1\mathbf{x}_0,\ldots,2\xi^TA_{\dim L}\mathbf{x}_0\right).
\end{equation*} If $D_\mathbf{x}\mathcal{F}_{(\mathbf{x}_0,L)}$ were not surjective, then there would exist
$0\neq A\in L$ such that $\xi^TA\mathbf{x}_0=0$ for every $\xi\in\mathbb R^n$, hence
$A\mathbf{x}_0=0$, contradicting the invertibility of $A$. Thus $D_\mathbf{x}\mathcal{F}_{(\mathbf{x}_0,L)}$ is surjective. By the implicit function
theorem, for every $\tilde{L}$ sufficiently close to $L$ there exists
$\mathbf{x}(\tilde{L})$ close to $\mathbf{x}_0$ such that
\begin{equation}\label{equ-persistence of real point}    
  \mathcal{F}(\mathbf{x}(\tilde{L}),\tilde{L})=0\Longleftrightarrow  \forall A\in \tilde{L},\, \mathbf{x}(\tilde{L})^TA\mathbf{x}(\tilde{L})=0
     \Longleftrightarrow [\mathbf{x}(\tilde{L})\otimes \mathbf{x}(\tilde{L})]\in \bZ\cap\bP(\tilde{L})^\vee,
\end{equation}
Thus this real point $[\mathbf{x}(\tilde{L})\otimes \mathbf{x}(\tilde{L})]$ lies near $[\mathbf{x}_0\otimes \mathbf{x}_0]$. After shrinking the neighborhood, the corresponding differential remains
surjective at $\mathbf{x}(\tilde L)$; hence $[\mathbf{x}(\tilde L)]$ is a smooth real point
of $X_{\tilde L_{\mathbb C}}$.

Next, observe that $\bP(\widetilde L)\cap\bY=\varnothing$ is an open condition. Indeed, $\bY$ is closed, after shrinking the
neighborhood of $L$ in the real Grassmannian, every nearby same dimensional
$\tilde L$ still satisfies
    $\bP(\tilde{L})\cap \bY=\varnothing$ .\par
For generic $L'\subset\Rsn$  in a sufficiently small neighborhood of $L$,  the complex quadrics associated with $L'_{\mathbb{C}}:=L'\otimes_\mathbb{R}\mathbb{C}\subset\mathbb{C}^{n\times n}_{\mathrm{sym}}$ cut out a
smooth complete intersection
\begin{equation*}
    X_{L'_{\mathbb{C}}}=\{[\mathbf{x}]\in\mathbb P^{n-1}_{\mathbb{C}}\mid \mathbf{x}^TA\mathbf{x}=0\ \text{for all }A\in L'_{\mathbb{C}}\}\subset \mathbb P^{n-1}_{\mathbb{C}}.
\end{equation*}  
Since $\dim X_{L'_{\mathbb{C}}}=n-1-\dim L'=n-1-\dim L$ which is greater than $0$ when $n\geq 3$, the  $X_{L'_{\mathbb{C}}}$ is irreducible. By the persistence result \eqref{equ-persistence of real point}, $X_{L'_{\mathbb{C}}}$ contains a real smooth point. Since $X_{L'_{\mathbb{C}}}$ is smooth and irreducible for a generic choice of $L'$, its real locus is Zariski dense in $X_{L'_{\mathbb{C}}}$.\par

The following is the standard consequence of complete intersections. Since $X_{L'_{\mathbb{C}}}$ is a complete intersection for a generic choice of $L'$,  every quadratic equation vanishing on
$X_{L'_{\mathbb C}}$ is a linear combination of quadrics $\mathbf{x}^TA\mathbf{x}$ with $A\in L'_{\mathbb{C}}$. Consequently,
if $\mathbf{x}^TB\mathbf{x}$ vanishes on $X_{L'_{\mathbb C}}$, then $B\in L'_{\mathbb C}$. 

%Consequently, $\bZ_{\mathbb{C}}\cap \bP(L'_{\mathbb{C}})^\vee$ is nondegenerate in $\bP(L'_{\mathbb{C}})^\vee$. Indeed, if $\bZ_{\mathbb{C}}\cap \bP(L'_{\mathbb{C}})^\vee$ were contained in a hyperplane of $\bP(L'_{\mathbb{C}})^\vee$, then the corresponding quadric would vanish on $X_{L'_{\mathbb{C}}}$, and hence would lie in $L'_{\mathbb{C}}$; this would make the hyperplane trivial on $\bP(L'_{\mathbb{C}})^\vee$, a contradiction.

\par
It remains to prove the  nondegeneracy of $\bZ\cap \bP(L')^\vee$ in $\bP(L')^\vee$. Suppose, for contradiction, that $\bZ(\mathbb R)\cap\bP(L')^\vee$ is degenerate in $\bP(L')^\vee$. Then there exists an ambient hyperplane $ H_B\subset \bP(\Rsn)^\vee$,
defined by a real symmetric matrix $B\in\Rsn$, namely $(H_B)^\vee=[B]\in\bP(\Rsn)$, satisfying $\bP(L')^\vee\not\subset H_B$ and
    $\bZ(\mathbb R)\cap\bP(L')^\vee\subset H_B$.
Equivalently, 
\begin{equation*}
\text{
$\mathbf{x}^TB\mathbf{x}=0$ for every real point
$[\mathbf{x}\otimes \mathbf{x}]\in \bZ(\mathbb R)\cap\bP(L')^\vee$.}\end{equation*}
 Since this real locus is Zariski dense
in $X_{L'_{\mathbb{C}}}$, the quadric $\mathbf{x}^TB\mathbf{x}$ vanishes on all of
$X_{L'_{\mathbb{C}}}$.
By the consequence of complete intersection above, this implies $B\in L'_{\mathbb{C}}$.
Since $B$ is real and $L'$ is a real subspace, we obtain    $B\in L'$, equivalently in dual space it is $\bP(L')^\vee\subset H_B$, contradicting the choice of $H_B$.
Hence $\bZ \cap \bP(L')^\vee$ is  nondegenerate.
\end{proof}

\bigskip

It is natural to ask what happens in the exceptional dimensions $2,4,8,16$ when
the dimension of the underlying invertible subspace of $\Rsn$ is increased from
$\rho(\frac{1}{2} n)$ to the maximal possible value $\rho(\frac{1}{2} n)+1$.  More
precisely, must the intersection $\bZ\cap\bL^\vee$ then be empty?  These
dimensions are closely related to the real division algebras
$\mathbb R,\mathbb{C},\mathbb H,\mathbb O$.  Together with the phenomenon observed
for $n=2$ in Example~\ref{CS example}, this suggests the following conjecture.
    
    \begin{conj}[Quadrics base locus conjecture]\label{conjecture}
For $n=2,4,8,16$,
  let $W_{n,\mathrm{sym}}$ be a linear subspace of $\mathbb{R}_{\mathrm{sym}}^{n\times n}$ with $\dim W_{n,\mathrm{sym}}=\frac{1}{2}n+1$, such that every nonzero element of $W_{n,\mathrm{sym}}$ is invertible. Take any basis $\{A_1,\cdots, A_{\frac{1}{2}n+1}\}$ of $W_{n,\mathrm{sym}}$.
       Then, 
       \begin{equation*}
           \bx^T A_1\bx=\cdots=
           \bx^T A_{\frac{1}{2}n+1}\bx=0
       \end{equation*}
        has no solution for $\bx \in \mathbb{R}^n-\{0\}$.
    \end{conj}

Recall $\rho(\frac{1}{2}n)+1=\frac{1}{2}n+1$ for $n=2,4,8,16$. Although the conjecture is stated using a basis of $W_{n,\mathrm{\mathrm{sym}}}$, the property is independent of this choice.  \par

 The conjecture is in fact known for the $n=2, 4$ cases as in the following propositions, but we choose to keep those cases there to emphasize the exceptional pattern $2,4,8,16$ and its relation to $8$-fold periodicity.

\begin{prop}\label{n=2}
    Conjecture \ref{conjecture} holds for $n=2$.
\end{prop}
\begin{proof}
    Suppose that $\bx_0\in\mathbb{R}^2$ satisfies $\bx_0^T A_1\bx_0=0$, $\bx_0^T A_2\bx_0=0$ with $\{A_1,A_2\}$ being a basis of $W_{2,\mathrm{sym}}$.   
 Then   the vectors $A_1\bx_0$ and $A_2\bx_0$ are both perpendicular to $\bx_0$ in $\mathbb{R}^2$,  $A_1\bx_0$ and $A_2\bx_0$ are collinear. This implies the existence of $(\lambda_1, \lambda_2)\neq (0,0)$ not all zero but $(\lambda_1A_1+\lambda_2A_2)\bx_0=0$. This contradicts the assumption that $(\lambda_1A_1+\lambda_2A_2)\in W_{2,\mathrm{sym}}$ is invertible.  
\end{proof}

\begin{prop}\label{n=4}
    Conjecture \ref{conjecture} holds for $n=4$.
\end{prop}

This $n=4$ case is known in \cite[last line of Table 1]{PLAUMANN_2011712}, which is due to the  research on the 28 bitangents of the quartic curves in $\bP^2$ (both complex and real). The following argument, suggested to us by Sergey Galkin, gives another proof using Euler characteristics.
\begin{proof}
    We consider the incidence correspondence
\begin{equation*}
\mathcal Y:=\left\{([M],[\mathbf x])\in \bP(W_{4,\mathrm{sym}})\times \mathbb R\bP^3
\ \middle|\ 
M=k_1A_1+k_2A_2+k_3A_3,\ \mathbf x^TM\mathbf x=0
\right\}.
\end{equation*}
Here $\bP(W_{4,\mathrm{\mathrm{sym}}})\cong \mathbb R \bP^2$. The first projection
$\mathcal Y\to \bP(W_{4,\mathrm{\mathrm{sym}}})$ has fibers given by smooth quadrics in
$\mathbb R \bP^3$. By Corollary~\ref{coro-even-dim-signature}, every nonzero element of
$W_{4,\mathrm{\mathrm{sym}}}$ has signature $(2,2)$. Hence each fiber is homeomorphic to
$\mathbb R \bP^1\times \mathbb R \bP^1$. Therefore 
\begin{equation*}
\chi(\mathcal Y)=\chi(\mathbb R \bP^2)\cdot\chi(\mathbb R \bP^1\times\mathbb R \bP^1)=1\cdot 0=0.
\end{equation*}

On the other hand, consider the second projection $\mathcal Y\to \mathbb R \bP^3$. For
$[\mathbf x]\in\mathbb R \bP^3$, the equation $\mathbf x^TM\mathbf x=0$ becomes the linear
equation
$\sum_{i=1}^3(\mathbf x^TA_i\mathbf x)k_i=0$
in the homogeneous coordinates $[k_1:k_2:k_3]$ of $\bP(W_{4,\mathrm{\mathrm{sym}}})\cong
\mathbb R \bP^2$. The exceptional locus is the base locus
$
\mathcal B:=\{[\mathbf x]\in\mathbb R \bP^3\mid
\mathbf x^TA_1\mathbf x=\mathbf x^TA_2\mathbf x=\mathbf x^TA_3\mathbf x=0\}.$

Notice that  $\bZ$ intersects $\bL^\vee$ transversally by Proposition \ref{prop-Z intersect L^vee transversally}, thus $\bZ\cap \bL^\vee$ is zero-dimensional, if nonempty. Meanwhile, $\mathcal{B}\subset\RP^3$ is the preimage of $\bZ\cap \bL^\vee$ under Veronese embedding, thus $\dim \mathcal{B}\leq 0$.

 Recall that $\chi(\mathcal{Y})=0$, we then have the following
 \begin{equation*}
     0=\chi(\mathcal{Y})=\chi(\mathcal{B})\cdot\chi(\RP^2)+\chi(\RP^3-\mathcal{B})\cdot \chi(\RP^1)=\chi(\mathcal{B})\cdot 1+0=\chi(\mathcal{B}).
 \end{equation*} 
 Since $\mathcal{B}$ is an algebraic subset of $\mathbb{RP}^3$ with
$\dim\mathcal{B}\leq 0$, it is finite. Hence
$\chi(\mathcal{B})=|\mathcal{B}|$, and therefore $\mathcal{B}=\varnothing$.
\end{proof}

%\begin{rema}
%The equations in Conjecture~\ref{conjecture} define the base locus of the linear system of quadrics associated with $W_{n,\mathrm{\mathrm{sym}}}$. Equivalently, this base locus is the common intersection of the null cones of the quadratic forms $\mathbf{x}^TA\mathbf{x}$, $A\in W_{n,\mathrm{\mathrm{sym}}}$. Thus the conjecture can be viewed as a statement about base loci of maximal invertible spaces of real quadratic forms.
%\end{rema}
\begin{proof}[proof of Proposition \ref{prop: answer for problem 2}]
Propositions \ref{prop-construction and nonempty intersection} and
\ref{prop-nondegenerancy} give the construction attaining lower bounds  in the statement for all $n\geq 3$. For $n=2$, Example \ref{CS example} gives two rank-one symmetric matrices spanning $\bL^\vee$, while Proposition \ref{n=2} gives the upper bound. \par

For the upper bound, Theorem \ref{Thm ALP} of Adams--Lax--Phillips
\cite{ALP_1965}, together with Proposition \ref{prop-diagonal doubled}, shows
that the dimension is at most $\rho(\frac{1}{2} n)+1$ for every $n$.  Propositions
\ref{n=2} and \ref{n=4} show that, in the cases $n=2,4$, this upper bound must
drop by one.  The remaining possibilities for $n=8,16$ are therefore precisely
$\rho(\frac{1}{2} n)+1$ or $\rho(\frac{1}{2} n)$.\qedhere
\end{proof}

We conclude this section by stating the proof of the decomposition lemma.
\begin{proof}[proof of Main Lemma \ref{Decomposition Lemma}]

Let $L\subset\Rsn$ be the subspace that  achieves the maximum possible dimension for all $n \notin \{8, 16\}$, which has dimension $\frac{n(n+1)}{2}-\Xi_n$.

By Lemma \ref{elimination lemma}, choose $\check{\Phi}\in C^{j+1,\alpha}(\overline{\Omega},\mathbb{R}^n)$ such that $\Pi_L (D+\sym \nabla\check{\Phi})=0$. Let $\hat{D}:=D+\sym\nabla\check{\Phi}$, thus $\hat{D}$ takes values in $L^\perp$, namely $\hat{D}\in C^{j,\alpha}(\overline{\Omega},L^\perp)$. We then apply Lemma \ref{nonnegative coefficient lemma} to $\hat{D}$ with $U=L^\perp$, then it gives $\hat{\Phi}\in C^{j+1,\alpha}(\overline{\Omega},\mathbb{R}^n)$ such that $\hat{D}+\sym\nabla \hat{\Phi}=\sum_{i=1}^{\Xi_n}a_i^2\xi_i\otimes\xi_i$. Hence $D+\sym\nabla(\check{\Phi}+\hat{\Phi})=\sum_{i=1}^{\Xi_n}a_i^2\xi_i\otimes\xi_i$, we obtain the decomposition \eqref{decomp1} by taking $\Phi=\check{\Phi}+\hat{\Phi}$. The required bounds follow immediately from Proposition \ref{prop between two lemmas} and Lemma \ref{nonnegative coefficient lemma}.

To prove the optimality of $\Xi_n$ for $n\neq 8,16$, suppose for contradiction that there exist $m_n < \Xi_n$ unit vectors $\xi_1,\dots,\xi_{m_n}$ such that the decomposition $D+\sym\nabla\Phi=\sum_{i=1}^{m_n} a_i^2\xi_i\otimes \xi_i$ and its accompanying estimates \eqref{equ-main lemma estimate} hold for any $D$. Let $U = \operatorname{span}\{\xi_1\otimes \xi_1,\dots,\xi_{m_n}\otimes \xi_{m_n}\}$ and set $L_{\mathrm{orth}}=U^\perp$. Consequently, $\Pi_{L_{\mathrm{orth}}}(D+\sym\nabla \Phi)=0$.

By assumption, $\dim L_{\mathrm{orth}}=\frac{n(n+1)}{2}-m_n > \frac{n(n+1)}{2}-\Xi_n$. By definition, $L_{\mathrm{orth}}{}^\perp=U$ is spanned by $\xi_i\otimes\xi_i$. Hence $L_{\mathrm{orth}}$ satisfies the condition \ref{condition 2 in affine problem} in Problem \ref{affine problem}. Therefore, since $L_{\mathrm{orth}}$ exceeds
the maximum dimension  in Proposition \ref{prop: answer for problem 2}, 
$L_{\mathrm{orth}}$ cannot also satisfy condition \ref{condition 1 in affine problem} in Problem \ref{affine problem},
namely,  there must exist a non-zero $A\in L_{\mathrm{orth}}\subset\Rsn$ such that $A+\operatorname{diag}(A)$ is not invertible. However, by Proposition \ref{prop-oscillatory construction for D}, this  guarantees the existence of a  $D$ for which any $\Phi$ satisfying $\Pi_{L_{\mathrm{orth}}}(D+\sym\nabla\Phi)=0$, if such a $\Phi$ exists, violates the estimates \eqref{equ-main lemma estimate}. This contradiction proves that $\Xi_n$ is minimal.
\end{proof}

\section{Proof of one stage induction and Theorem \ref{main theorem}}

Now we use convex integration method to construct the solution required in Theorem \ref{main theorem}. 
The proof here is well established, as the application of convex integration method to equation \eqref{ddagger} has matured in recent years. See, for example,
\cite{conti_h-principle_2012,lewicka_convex_2017,cao_very_2019,lewicka2025monge,cao_hirsch_inauen_2025,li_Qiu_very_2024}.

\subsection{ A quick start from \cite{cao_hirsch_inauen_2025}}\label{section:setp 1}

We consider the equation \eqref{ddagger}, namely
\begin{equation*}
    A =\frac{1}{2}\nabla v\otimes \nabla v +\sym \nabla w.
\end{equation*}
Applying the scaling technique in \cite{cao_hirsch_inauen_2025},  for the given $A\in C^2(\overline{\Omega},\Rsn),\,v^\flat\in C^0(\overline{\Omega}),\,w^\flat\in C^0(\overline{\Omega},\mathbb{R}^n)$, we fix a constant
\begin{equation*}
\tau:=|A
|+\lVert v^\flat \rVert_2^2+\|w^\flat\|_2^2+100>1,
\end{equation*}
where we apply extension and mollification  to assume $v^\flat\in C^\infty(\overline{\Omega}),\,w^\flat\in C^\infty(\overline{\Omega},\mathbb{R}^n)$, and let 
\begin{equation}\label{definition V0 W0}
    \overline{A}:=\delta_1\tau^{-1}A, \quad V_0:=\delta_1^{\frac{1}{2}}\tau^{-\frac{1}{2}}v^\flat,\quad  W_0:=\delta_1\tau^{-1}w^\flat.
\end{equation}
Then the $(\overline{A}, V_0, W_0)$ satisfies the initial condition in stage proposition \ref{prop:induction} (see \eqref{eq:deficit_q} for definition of $D_q$), thus a solution $(\underline{v},\underline{w})$ of 
\begin{equation*}
\overline{A}=\frac{1}{2}\nabla \underline{v}\otimes \nabla \underline{v} + \sym \nabla \underline{w}
\end{equation*}
is obtained via induction on stages in Proposition \ref{prop:induction}. Then $(v,w)=(\delta_1^{-\frac{1}{2}}\tau^{\frac{1}{2}}\underline{v},\delta_1^{-1}\tau \underline{w})$ solves \eqref{ddagger}, and we postpone the verification of $\|v-v^\flat\|_0<\epsilon$ to the end of the proof.

\subsection{ Induction on stages}

Set $a>1, \ 1<b<2,$ and $  c>0$ be three real positive numbers to be determined as parameters, and take 
\begin{equation}\label{definition of delta and lambda}
  \begin{array}{lllll}
   \delta_q:=a^{-b^q}<1,& \lambda_q:=a^{cb^q}>1, &C_*>1& K>1,& \EC>C_*^2+\|A\|_1,
\end{array}
  \end{equation}
  where $C_*$, $K$, $\EC$ would be constants.\par
  Recall the definition of $\Xi_n$ in \eqref{defi of Xi}. For later application, we further denote

  \begin{equation}\label{level three definition}
   1<\frac{K\delta^{\frac{1}{2}}_q\lambda_q}{\delta^{\frac{1}{2}}_{q+1}}=:\mu_0\leq\mu_1\leq\cdots\leq \mu_{\Xi_n}:=\lambda_{q+1},\quad
      \mu_i:=\mu_0^{1-\frac{i}{\Xi_n}}\mu^{\frac{i}{\Xi_n}}_{\Xi_n},\quad l:=\frac{1}{C_*\mu_0}=\frac{\delta^{\frac{1}{2}}_{q+1}}{C_* K\delta^{\frac{1}{2}}_q\lambda_q}<1.
  \end{equation}
 All above terms will be determined later according to the requirements in the next proposition.

Denote the $q$-th stage deficit matrix by
\begin{equation}
    \label{eq:deficit_q}
    D_q:=\overline{A}-\frac{1}{2} \nabla V_q\otimes \nabla V_q-\sym\nabla W_q.
\end{equation}

The following proposition provides the crucial ``one-stage" induction as a common notion in subsequent works following Nash's construction \cite{Nash_C1}. We adhere to the clear exposition provided by Li-Qiu \cite{li_Qiu_very_2024} and include the proof here for completeness, as our presentation on this one-stage induction does not introduce any new results. Note, however, that our $q$-th stage deficit $D_q$ does not contain $\delta_{q+1}Id$ term---compare with (3.23) in \cite{li_Qiu_very_2024}---thanks to our main Lemma \ref{Decomposition Lemma}, in particular, the nonnegative coefficients lemma \ref{nonnegative coefficient lemma}.

%which is applicable to any $D\in C^{j,\alpha}(\overline{\Omega},\Rsn)$, enabling us to set the constant $\sigma$  that controls $\|D-Id\|_\alpha$ in \cite{li_Qiu_very_2024} to be $1$.
\begin{prop}[Stage]
    \label{prop:induction}
    There exist $a>1$, $1<b<2$, $c>0$ and   $K>1$, depending only on the fixed data of the scheme, such that, if $(V_q,W_q)\in C^2(\overline{\Omega})\times C^2(\overline{\Omega},\RR^n)$ satisfies 
    \begin{align}
        \|(V_q,W_q)\|_1&\leq\sqrt{K},\label{norm1 assumption} \\
        \|(V_q,W_q)\|_2&\leq K\delta_q^{\frac{1}{2}}\lambda_q,\label{norm 2 assumption}\\
        \|D_q\|_0&\leq\delta_{q+1}\label{norm D assumption},        
    \end{align}
    then we can construct $(V_{q+1},W_{q+1})\in C^2(\overline{\Omega})\times C^2(\overline{\Omega},\RR^n)$ with
    \begin{align}
\|(V_{q+1}-V_q,W_{q+1}-W_q)\|_0&\leq (\delta_{q+1}^{\frac{1}{2}}+1)\frac{\delta_{q+1}^{\frac{1}{2}}}{\delta_{q}^{\frac{1}{2}}\lambda_q},\label{conclusion 0}\\
        \|(V_{q+1}-V_q,W_{q+1}-W_q)\|_1&\leq K\delta_{q+1}^{\frac{1}{2}},\label{conclusion 1}\\
        \|(V_{q+1},W_{q+1})\|_2&\leq K\delta_{q+1}^{\frac{1}{2}}\lambda_{q+1},\label{conclusion 2}\\
        \|D_{q+1}\|_0&\leq\delta_{q+2}, \label{conclusion 3}       
    \end{align}
\end{prop}

\begin{proof}
\noindent\textbf{1. Mollification.}  Consider the mollification $v_0:=V_q*\phi_l$, $w_0:=W_q*\phi_l$, and denote 
    \begin{equation*}
        \tilde{D}:=\overline{A}*\phi_l-\frac{1}{2}\nabla v_0\otimes \nabla v_0-\sym \nabla w_0.
    \end{equation*}
 Notice by property of mollification, $D_q*\phi_l=\overline{A}*\phi_l-\frac{1}{2}(\nabla V_q\otimes\nabla V_q)*\phi_l-\sym\nabla w_0$, then
    \begin{equation*}
        \tilde{D}=D_q*\phi_l+\frac{1}{2}\left((\nabla V_q\otimes\nabla V_q)*\phi_l-\nabla v_0\otimes\nabla v_0\right).
    \end{equation*}

With the definition of $\mu_0$ in \eqref{level three definition}, the estimates on mollification in Lemma \ref{Mollification Lemma} and the initial assumptions \eqref{norm1 assumption} \eqref{norm 2 assumption} yield 
\begin{equation}\label{v0-vq:0 norm}
    \|v_0-V_q,w_0-W_q\|_0\lesssim\|(V_q,W_q)\|_1 l\lesssim \sqrt{K} \cdot l,
\end{equation}
\begin{equation}\label{v0-Vq}
    \lVert(v_0-V_q,w_0-W_q)\rVert_1\lesssim \lVert(V_q,W_q)\rVert_2 l\lesssim K\delta_q^{\frac{1}{2}}\lambda_ql\lesssim \delta_{q+1}^{\frac{1}{2}}\mu_0l\lesssim \delta_{q+1}^{\frac{1}{2}},
\end{equation}
\begin{equation}\label{v0 w0 norm 2+k}
    \lVert (v_0,w_0)\rVert_{2+k}\lesssim\lVert (V_q,W_q)\rVert_2l^{-k}\lesssim\delta_{q+1}^{\frac{1}{2}}\mu_0l^{-k}, \quad k=0,1.
\end{equation}
By using the estimates of mollification in Lemma \ref{Mollification Lemma}, the assumptions  \eqref{norm1 assumption} \eqref{norm D assumption}, and the definition $\mu_0=\frac{K\delta_q^{\frac{1}{2}}\lambda_q}{\delta_{q+1}^{\frac{1}{2}}}$ in \eqref{level three definition}, one has
\begin{align*}
    \rVert\tilde{D}\rVert_0&\leq \rVert D_q*\phi_l\rVert_0+\frac{1}{2}\rVert(\nabla V_q\otimes\nabla V_q)*\phi_l-\nabla v_0\otimes\nabla v_0\lVert_0\\
    &\leq \lVert D_q\rVert_0+Cl^2\lVert V_q\rVert_2^2\\
    &\leq \delta_{q+1}+C(\mu_0l)^2\delta_{q+1}.
\end{align*}
Recall the definition $l=\frac{1}{C_* \mu_0}$ in \eqref{level three definition}, then for some $C_*>1$, we have 
\begin{equation*}
    \lVert\tilde{D}\lVert_{0}\leq  2  \delta_{q+1}, \quad \text{namely} \quad
    \lVert\frac{\tilde{D}}{\delta_{q+1}}\lVert_{0}\leq  2 .
\end{equation*}
Similar estimates for higher derivatives give
\begin{equation*}
    \lVert \tilde{D}\lVert_{k}\leq C\delta_{q+1}l^{-k}, \quad \text{ for }1\leq k\leq 3.
\end{equation*}

\noindent\textbf{2. Decomposition.} Applying Lemma \ref{Decomposition Lemma} to $\frac{\tilde{D}}{\delta_{q+1}}$ gives us  $\tilde{\Phi}$ and $\tilde{a}_i$ such that
$\frac{\tilde{D}}{\delta_{q+1}}=-\sym(\nabla \frac{\tilde{\Phi}}{\delta_{q+1}})+\sum_{i=1}^{\Xi_n}\frac{\tilde{a}_i^2}{\delta_{q+1}}\xi_i\otimes\xi_i,$ namely
\begin{equation}\label{D tilde decomposition} 
\overline{A}*\phi_l-\frac{1}{2}\nabla v_0\otimes\nabla v_0-\sym\nabla  w_0=-\sym(\nabla \tilde{\Phi})+\sum_{i=1}^{\Xi_n}\tilde{a}_i^2\xi_i\otimes\xi_i. 
\end{equation}
Next, we apply a similar technique as in Lemma \ref{nonnegative coefficient lemma}, adding a positive constant to $\tilde{a}_i^2$ to obtain working estimates of $\|\nabla^ka_i\|_0$. Since $\|\tilde{D}\|_0\leq2 \delta_{q+1},\,\,\|\tilde{D}\|_k\leq C\delta_{q+1} l^{-k}$, we let $a_i=\sqrt{\tilde{a}_i^2+\delta_{q+1}}$, and $\Phi=\tilde{\Phi}+\delta_{q+1}(\sum_{i=1}^{\Xi_n}\xi_i\otimes \xi_i)x$ for $x\in \overline{\Omega}$, then $-\sym(\nabla \tilde{\Phi})+\sum\tilde{a}_i^2\xi_i\otimes\xi_i=-\sym(\nabla\Phi)+\sum a_i^2 \xi_i\otimes\xi_i $.

Since the tensors $\xi_i\otimes\xi_i$ are fixed and linearly independent, the estimates in Lemma \ref{Decomposition Lemma} imply
\begin{equation*}
\|\tilde a_i^2\|_0\lesssim \delta_{q+1},
\qquad
\|\tilde a_i^2\|_k\lesssim \delta_{q+1}l^{-k},
\quad 1\leq k\leq 3.
\end{equation*}
The definition of $a_i$, $\Phi$ and chain rule gives
\begin{equation}\label{a_i estimates 0}
\|a_i\|_0
\lesssim
\delta_{q+1}^{1/2},
\end{equation}
\begin{equation}\label{a_i estimates 1 to 3}
\|\nabla^k a_i\|_0
\lesssim
\delta_{q+1}^{1/2}l^{-k},
\qquad 1\leq k\leq 3,
\end{equation}
\begin{equation}\label{Phi estimates 1 to 2}
\|\Phi\|_k
\lesssim
\|\tilde\Phi\|_k+\delta_{q+1}
\lesssim
\delta_{q+1}\left\|\frac{\tilde D}{\delta_{q+1}}\right\|_{k-1}+\delta_{q+1}
\lesssim
\delta_{q+1}l^{1-k},
\qquad 1\leq k\leq 2 .
\end{equation}

\noindent\textbf{3. Construction of $i$th step.} Now we start the induction over the steps  $1\leq i\leq \Xi_n$ in one-stage,  with $\Xi_n<\frac{n(n+1)}{2} $ as defined in Definition \ref{defi of Xi}.
\begin{equation}\label{corrugation in each step}
\begin{aligned}
    &v_i:=v_{i-1}+\frac{1}{\mu_i}\Gamma_1(a_i, \mu_ix\cdot \xi_i),\\
    & 
w_i:=\begin{dcases}
        w_{0}-\Phi-\frac{1}{\mu_1}\Gamma_1(a_1,\mu_1x\cdot \xi_1)\nabla v_{0}+\frac{1}{\mu_1}\Gamma_2(a_1,\mu_1x\cdot \xi_1)\xi_1 & \text{for } i=1,\\
        w_{i-1}-\frac{1}{\mu_i}\Gamma_1(a_i,\mu_ix\cdot \xi_i)\nabla v_{i-1}+\frac{1}{\mu_i}\Gamma_2(a_i,\mu_ix\cdot \xi_i)\xi_i & \text{for } 2\leq i\leq\Xi_n,
    \end{dcases}\\
    & (V_{q+1},W_{q+1}):=(v_{\Xi_n},w_{\Xi_n}).
\end{aligned}
\end{equation}

The estimates for $v_i$
  are obtained as follows. For $0\leq k\leq 3$, by the estimates on $\Gamma_1$ in \eqref{Gamma estimates} and on $a_i$ in \eqref{a_i estimates 0}\eqref{a_i estimates 1 to 3}, we have
\begin{equation*}
    \begin{aligned}
        \lVert v_i-v_{i-1}\rVert_k&\lesssim \frac{1}{\mu_i}(\mu_i^k\lVert a_i\rVert_0+\mu_i^{k-1}\lVert a_i\rVert_1+\cdots+\lVert a_i\rVert_k )\\
        &\lesssim\frac{1}{\mu_i}(\mu_i^k\delta_{q+1}^{\frac{1}{2}}+\mu_i^{k-1}\delta_{q+1}^{\frac{1}{2}}l^{-1}+\cdots+\delta^{\frac{1}{2}}_{q+1}l^{-k}
        ).
    \end{aligned}
\end{equation*}
If we require
\begin{equation}\label{dagger_C_*_LEQ}
    l^{-1}= C_*\mu_0\leq \mu_1 \tag{$\dagger_{C_*\leq}$},
\end{equation}
 then
\begin{equation}\label{vi-vi-1}
    \lVert v_i-v_{i-1}\rVert_k\lesssim \delta_{q+1}^{\frac{1}{2}}\mu_i^{k-1}.
\end{equation}
Consequently, together with $\lVert v_0\rVert_{3}\lesssim \delta_{q+1}^{\frac{1}{2}}\mu_0l^{-1}$ as in \eqref{v0 w0 norm 2+k}, we have
\begin{equation}\label{estimates of v_i_k}
\begin{gathered}
    \lVert\nabla v_i\rVert_0\lesssim \lVert \nabla V_q\rVert_0+\delta_{q+1}^{\frac{1}{2}}\lesssim \sqrt{K},\\
    \lVert\nabla^2v_i\rVert_0\leq \|\nabla^2v_0\|_0+\sum_{j=1}^i\|\nabla^2(v_j-v_{j-1})\|_0
  \lesssim \delta_{q+1}^{\frac{1}{2}}\mu_0+\delta_{q+1}^{\frac{1}{2}}(\mu_1+\cdots+\mu_i)\lesssim \delta_{q+1}^{\frac{1}{2}}\mu_i,\\
  \lVert\nabla^3v_i\rVert_0\leq \|\nabla^3v_0\|_0+\sum_{j=1}^i\|\nabla^3(v_j-v_{j-1})\|_0
  \lesssim \delta_{q+1}^{\frac{1}{2}}\mu_0l^{-1}+\delta_{q+1}^{\frac{1}{2}}(\mu_1^2+\cdots+\mu_i^2)\lesssim \delta_{q+1}^{\frac{1}{2}}\mu_i^2.
\end{gathered}
\end{equation}

Next, we establish the estimates for $w_i$. For $0\leq k\leq 2$, by requirement $l^{-1}\leq\mu_1$ , estimates on $\Gamma_1$ and $\Gamma_2$ in \eqref{Gamma estimates} and estimates on  $a_i$, $\nabla v_{i-1}$ in \eqref{a_i estimates 0} \eqref{a_i estimates 1 to 3} \eqref{estimates of v_i_k}, and $\Phi$ in \eqref{Phi estimates 1 to 2}, we have
\begin{equation}   \label{wi-wi-1} 
\begin{aligned}
\|w_i-w_{i-1}\|_k\lesssim &\frac{1}{\mu_i}\sum_{j_1,j_2,j_3\geq0}^{j_1+j_2+j_3=k}\|a_i\|_{j_1}\|\nabla v_{i-1}\|_{j_2}\mu_i^{j_3}+\frac{1}{\mu_i}(\|a_i^2\|_0\mu_i^k+\cdots\|a_i^2\|_k)+\|\Phi\|_k\\
\lesssim&\frac{1}{\mu_i}(\sqrt{K}\sum_{j_1,j_3\geq0}^{j_1+j_3=k}\delta_{q+1}^{\frac{1}{2}}l^{-j_1}\mu_i^{j_3}+\sum_{j_1,j_3\geq0,j_2\geq 1}^{j_1+j_2+j_3=k}
\delta_{q+1}^{\frac{1}{2}}l^{-j_1}
\delta_{q+1}^{\frac{1}{2}}\mu_i^{j_2}\mu_i^{j_3})\\
&+\frac{1}{\mu_i}(\delta_{q+1}\mu_i^k+\cdots+\delta_{q+1}l^{-k})+\delta_{q+1}l^{1-k}\\
\lesssim& \sqrt{K}\delta_{q+1}^{\frac{1}{2}}\mu_i^{k-1}    .
\end{aligned}
\end{equation}
Although only $w_1-w_0$ contains $\Phi$, its contribution is absorbed in the same estimate. \par
Combining \eqref{v0-vq:0 norm} \eqref{v0-Vq} \eqref{v0 w0 norm 2+k} \eqref{vi-vi-1} \eqref{wi-wi-1} and the definition of $\mu_i$ in \eqref{level three definition}, we conclude that

\begin{equation*}
    \begin{aligned}
        \|(V_{q+1}-V_q,W_{q+1}-W_q)\|_0&\lesssim \|(v_0-V_q,w_0-W_q)\|_0+\sum_{i=1}^{\Xi_n}(\|v_i-v_{i-1}\|_0+\|w_i-w_{i-1}\|_0)\\
        &\lesssim  \frac{\sqrt{K}}{ C_*\mu_0}+ \sqrt{K}\delta_{q+1}^{\frac{1}{2}}(\sum_{i=1}^{\Xi_n}\frac{1}{\mu_i})\leq C\left(\frac{1}{\sqrt K}+\frac1{\sqrt K}\delta_{q+1}^{\frac{1}{2}}\right)
\frac{\delta_{q+1}^{\frac{1}{2}}}
{\delta_q^{\frac{1}{2}}\lambda_q}\leq K(\delta_{q+1}^{\frac{1}{2}}+1)\frac{\delta_{q+1}^{\frac{1}{2}}}{\delta_{q}^{\frac{1}{2}}\lambda_q},
    \end{aligned}
\end{equation*}

\begin{equation*}
    \begin{aligned}
        \|(V_{q+1}-V_q,W_{q+1}-W_q)\|_1&\lesssim \|(v_0-V_q,w_0-W_q)\|_1+\sum_{i=1}^{\Xi_n}(\|v_i-v_{i-1}\|_1+\|w_i-w_{i-1}\|_1)\lesssim \sqrt{K}\delta_{q+1}^{\frac{1}{2}}\leq K\delta_{q+1}^{\frac{1}{2}},
    \end{aligned}
\end{equation*}
\begin{equation*}
\begin{aligned}
    \|(V_{q+1},W_{q+1})\|_2&\lesssim\|(v_0,w_0)\|_2+\sum_{i=1}^{\Xi_n}(\|v_i-v_{i-1}\|_2+\|w_i-w_{i-1}\|_2)\lesssim\sqrt{K}\delta_{q+1}^{\frac{1}{2}}(\sum_{i=0}^{\Xi_n}\mu_i)\leq K\delta_{q+1}^{\frac{1}{2}}\lambda_{q+1},
\end{aligned}
\end{equation*}
as long as $K>1$ large enough to cover all the $\lesssim$ . Thus the conclusions \eqref{conclusion 0}, \eqref{conclusion 1} and \eqref{conclusion 2} arrive.\par

\noindent\textbf{4. Estimation of error $D_{q+1}$.}    Recall the definition of $D_{q+1}$ as in \eqref{eq:deficit_q}, the equation on $a_i$ and $\Phi$  in \eqref{D tilde decomposition} and the corrugation in each step in \eqref{corrugation in each step}, we have 
    \begin{align*}
        D_{q+1}=&\overline{A}-\frac{1}{2}\nabla V_{q+1}\otimes \nabla V_{q+1}-\sym \nabla W_{q+1} -\overline{A}*\phi_l+\frac{1}{2}\nabla v_0\otimes \nabla v_0+\sym\nabla( w_0- \Phi) +\sum_{i=1}^{\Xi_n}a_i^2\xi_i\otimes\xi_i
        \\
        =&\sum_{i=1}^{\Xi_n}\left(a_i^2\xi_i\otimes\xi_i-\frac{1}{2}(\nabla v_{i}\otimes \nabla v_{i}-\nabla v_{i-1}\otimes \nabla v_{i-1})+\sym \nabla(\frac{1}{\mu_i}\Gamma_1\nabla v_{i-1}-\frac{1}{\mu_i}\Gamma_2\xi_i)\right)\\
        &+\overline{A} -\overline{A}*\phi_l+\sym\nabla \Phi-\sym\nabla \Phi\\
        =&\sum_{i=1}^{\Xi_n}\left(a_i^2\xi_i\otimes\xi_i-\sym(\frac{1}{\mu_i}\nabla\Gamma_1\otimes\nabla v_{i-1})-\frac{1}{2\mu_i^2}\nabla\Gamma_1\otimes \nabla\Gamma_1+\sym \nabla(\frac{1}{\mu_i}\Gamma_1\nabla v_{i-1}-\frac{1}{\mu_i}\Gamma_2\xi_i) \right)+\overline{A}-\overline{A}*\phi_l
        \\ 
        =&\sum_{i=1}^{\Xi_n}\left(a_i^2\xi_i\otimes\xi_i+
        \frac{1}{\mu_i}\Gamma_1\nabla^2 v_{i-1}-\frac{1}{2\mu_i^2}\nabla\Gamma_1\otimes \nabla\Gamma_1-\frac{1}{\mu_i}\sym (\nabla\Gamma_2\otimes\xi_i) \right)+\overline{A}-\overline{A}*\phi_l
    \end{align*}
Here and throughout the following, we use the shorthand $\Gamma_1$ (resp. $\Gamma_2$) to denote $\Gamma_1(a_i,\mu_ix\cdot\xi_i)$ (resp. $\Gamma_2(a_i,\mu_ix\cdot\xi_i)$).
We denote the $i$-th step error term as
\begin{equation}
\mathcal{E}_i:=a_i^2\xi_i\otimes\xi_i+
        \frac{1}{\mu_i}\Gamma_1\nabla^2 v_{i-1}-\frac{1}{2\mu_i^2}\nabla\Gamma_1\otimes \nabla\Gamma_1-\frac{1}{\mu_i}\sym (\nabla\Gamma_2\otimes\xi_i).    
\end{equation}
Notice that $\nabla\Gamma_k=\partial_s\Gamma_k\nabla a_i+\partial_t\Gamma_k\mu_i\xi_i$. Using the identity $\frac{1}{2}|\partial_t \Gamma_1|^2+\partial_t\Gamma_2=a_i^2$ as in  \eqref{dagger_Gamma}, the coefficients of $\xi_i\otimes\xi_i$ cancel, we have 
\begin{equation*}
\begin{aligned}
\mathcal{E}_i=&a_i^2\xi_i\otimes\xi_i+
        \frac{1}{\mu_i}\Gamma_1\nabla^2 v_{i-1}-\frac{1}{2\mu_i^2}|\partial_s\Gamma_1|^2\nabla a_i\otimes \nabla a_i-\frac{\partial_s\Gamma_1\partial_t\Gamma_1}{\mu_i}\sym (\nabla a_i\otimes\xi_i)     -\frac{1}{2}|\partial_t\Gamma_1|^2\xi_i\otimes\xi_i\\
        &-\partial_t\Gamma_2\xi_i\otimes\xi_i-\frac{\partial_s\Gamma_2}{\mu_i}\sym (\nabla a_i\otimes\xi_i)\\
        =&\frac{1}{\mu_i}\Gamma_1\nabla^2 v_{i-1}-\frac{1}{2\mu_i^2}|\partial_s\Gamma_1|^2\nabla a_i\otimes \nabla a_i-\frac{\partial_s\Gamma_2+\partial_s\Gamma_1\partial_t\Gamma_1}{\mu_i}\sym (\nabla a_i\otimes\xi_i).
\end{aligned}
\end{equation*}
Now using the  estimates of $\Gamma_{1}$ and $\Gamma_2$ in \eqref{Gamma estimates},  $\nabla^2v_{i-1}$ in \eqref{estimates of v_i_k} and  $a_i$ in \eqref{a_i estimates 0}, we have
\begin{equation*}
\begin{aligned}
\|\mathcal{E}_i\|_0\lesssim &\frac{1}{\mu_i}\|\Gamma_1\|_0\|\nabla^2v_{i-1}\|_0+\frac{1}{\mu_i^2}\|\partial_s\Gamma_1\|_0^2\|\nabla a_i\|_0^2+\frac{1}{\mu_i}(\|\partial_s\Gamma_2\|_0+\|\partial_s\Gamma_1\|_0\|\partial_t\Gamma_1\|_0)\|\nabla a_i\|_0\\
\lesssim&\frac{1}{\mu_i}\delta_{q+1}^{\frac{1}{2}}\cdot \delta_{q+1}^{\frac{1}{2}}\mu_{i-1}+\frac{1}{\mu_i^2}(\delta_{q+1}^{\frac{1}{2}}l^{-1})^2+\frac{1}{\mu_i}\delta_{q+1}^{\frac{1}{2}}\cdot \delta_{q+1}^{\frac{1}{2}}l^{-1}\\
\lesssim& C_*^2\,\delta_{q+1}\frac{\mu_{i-1}}{\mu_i}.
\end{aligned}    
\end{equation*}
Due to the high regularity of $A$ and the estimate \eqref{f-f*phi}, we have $\|\overline{A}-\overline{A}*\phi_l\|_0\lesssim\|\overline{A}\|_1l\lesssim \|A\|_1l,$    
    thus 
    \begin{equation*}
        \|D_{q+1}\|_0=\|\overline{A}-\overline{A}*\phi_l+\sum_{i=1}^{\Xi_n}\mathcal{E}_i\|_0\lesssim \|A\|_1l+ C_*^2\delta_{q+1}\sum_{i=1}^{\Xi_n}\frac{\mu_{i-1}}{\mu_i}.
    \end{equation*}
Recall the definition $\mu_i=\mu_0^{1-\frac{i}{\Xi_n}}\mu_{\Xi_n}^{\frac{i}{\Xi_n}}$ as in \eqref{level three definition}, if we require the $a,b,c$ to satisfy

\begin{equation}\label{dagger_delta_q+2}
    l\leq \frac{ \delta_{q+2}}{\EC},\quad \frac{\mu_0}{\mu_{\Xi_n}}\leq(\frac{\delta_{q+2}}{\EC\delta_{q+1}})^{\Xi_n},
\tag{$\dagger_{\delta_{q+2}}$}\end{equation}
we will then get 
\begin{equation*}
    \|D_{q+1}\|_0\lesssim \|A\|_1 l+ C_*^2\delta_{q+1}(\frac{\mu_0}{\mu_{\Xi_n}})^{\frac{1}{\Xi_n}}\leq ( C_*^2+\|A\|_1)\frac{\delta_{q+2}}{\EC}.
\end{equation*}
Hence to get the conclusion  $\|D_{q+1}\|_0\leq  \delta_{q+2}$ \eqref{conclusion 3}, we require $\EC $ large enough to cover the $\lesssim$ along the estimates. Hereby we  conclude the proof.
\end{proof}

\subsection{ Conclusion}
 We are going to determine the parameters $a>1$, $1<b<2$, $c>0$, and compute the upper bound of $\alpha$ such that $(V_q,W_q)$ converges in $C^{1,\alpha}$. The requirements to be satisfied are \eqref{dagger_C_*_LEQ} \eqref{dagger_delta_q+2}:
\begin{equation*}
    C_* \mu_0\leq \mu_1,\quad \frac{1}{C_* \mu_0}\leq \frac{ \delta_{q+2}}{\EC},\quad\frac{\mu_0}{\mu_{\Xi_n}}\leq (\frac{\delta_{q+2}}{\EC \delta_{q+1}})^{\Xi_n},
\end{equation*}
   Recall by definition \eqref{level three definition} $\mu_0=K\delta_q^{\frac{1}{2}}\lambda_q{\delta_{q+1}^{-\frac{1}{2}}}= K\cdot a^{(\frac{b-1}{2}+c)b^q}$,  $\mu_{\Xi_n}=\lambda_{q+1}=a^{cb^{q+1}}$, $\mu_1=\mu_0^{\frac{\Xi_n-1}{\Xi_n}}\mu_{\Xi_n}^{\frac{1}{\Xi_n}}= (Ka^{(\frac{b-1}{2}+c)b^q})^{1-\frac{1}{\Xi_n}}(a^{cb^{q+1}})^{\frac{1}{\Xi_n}}$, then the above three inequalities become
\begin{align}
   C_* ^{\Xi_n} K&\leq a^{(c-\frac{1}{2})(b-1)b^q},\label{step3 ineq 1}\\
   \EC a^{b^{q+2}}&\leq  KC_*\cdot a^{(\frac{b-1}{2}+c)b^q},\label{step3 ineq 2}\\
    K\cdot a^{(\frac{1}{2}-c)(b-1)b^q}&\leq \EC^{-\Xi_n}a^{\Xi_n(b-b^2)b^q}.\label{step3 ineq 3}
\end{align}

For fixed $b,c$, the constants $C_*,K,\EC$ can be absorbed by choosing $a$ sufficiently large. Thus the above requirements are ensured by

\begin{equation*}
    c>b^2-\frac {b}{2}+\frac{1}{2},\qquad
c>\frac{1}{2}+\Xi_n b.
\end{equation*}
Taking $b>1$ sufficiently close to $1$ and then $c>\Xi_n+\frac{1}{2}$ sufficiently close to $\Xi_n+\frac{1}{2}$, the requirements are satisfied.\par

The last remaining requirement for running the induction is to ensure that the initial condition \eqref{norm1 assumption}, namely $\|(V_q,W_q)\|_1\leq \sqrt{K}$,  is satisfied for all $q\geq 0$ . From \eqref{definition V0 W0}, we observe that $\|(V_0,W_0)\|_1\leq C\delta_1^{\frac{1}{2}}  <\sqrt{K}/2$. Recalling  \eqref{conclusion 1} and the definition  $\delta_{q}=a^{-b^q}$ from \eqref{definition of delta and lambda}, we note that for any $K>1$ and $b>1$, there exists a large enough $a$ such that 
\begin{equation} 
\|(V_{q+1},W_{q+1})-(V_{q},W_{q})\|_1\leq K a^{-b^{q+1}/2}<\sqrt{K}\cdot(\frac{1}{2})^{q+1}\quad \text{for any } q\geq 0.
\end{equation} 
It then follows  that \eqref{norm1 assumption} holds for all $q\geq0$.

   Now consider 
   \begin{equation*}
   \begin{aligned}
       \|(V_{q+1},W_{q+1})-(V_q,W_q)\|_{1+\alpha}&\leq \|(V_{q+1},W_{q+1})-(V_q,W_q)\|_1^{1-\alpha}\|(V_{q+1},W_{q+1})-(V_q,W_q)\|_2^\alpha\\
       &\leq K \delta_{q+1}^\frac{1}{2}\lambda_{q+1}^\alpha\leq Ka^{(c\alpha-\frac{1}{2})b^{q+1}},
   \end{aligned}
   \end{equation*} 
   hence for $a$ sufficiently large,
    the sequence $(V_q,W_q)$ is Cauchy in $C^{1,\alpha}$ provided that
 \begin{equation}\label{final}
       c\alpha-\frac{1}{2}<0.
   \end{equation}

If we take $b>1$ close enough to $1$, $c>\Xi_n+\frac{1}{2}$ close enough to $\Xi_n+\frac{1}{2}$, then any $\alpha<\frac{1}{1+2\Xi_n}$ would satisfy \eqref{final}. Denote $(\underline{v},\underline{w}):=\lim_{q\rightarrow\infty}(V_q,W_{q})$, and recall the final solution to \eqref{ddagger} would be $(v,w)=(\delta_1^{-\frac{1}{2}}\tau^{\frac{1}{2}}\underline{v},\delta_1^{-1}\tau \underline{w})$ as discussed in Section \ref{section:setp 1}.
   For any $\epsilon>0$, using $1<b<2$, $c>\Xi_n+\frac{1}{2}$, the conclusion \eqref{conclusion 0}, and definitions $\delta_q=a^{-b^q}$, $\lambda_q=a^{cb^q}$ from \eqref{definition of delta and lambda}, we obtain
   \begin{align*}
       \|v-v^\flat\|_0&\leq \|\delta_1^{-\frac{1}{2}}\tau^{\frac{1}{2}}V_0-v^\flat\|_0+\sum_{q\geq 0}\delta_1^{-\frac{1}{2}}\tau^{\frac{1}{2}}\|V_{q+1}-V_q\|_0\\
       &\leq 2\tau^{\frac{1}{2}}\sum_{q\geq0}\frac{\delta_{q+1}^{\frac{1}{2}}}{\delta_1^{\frac{1}{2}}\delta_{q}^{\frac{1}{2}}\lambda_q}\leq 2\tau^{\frac{1}{2}}\sum_{q\geq0} a^{\frac{b}{2}}(a^{\frac{1-b-2c}{2}})^{b^q}\\
       &\leq 2\tau^{\frac{1}{2}}\sum_{q\geq0} a \cdot(a^{-\frac{2\Xi_n+1}{2}})^{1+q(b-1)}< 2\tau^{\frac{1}{2}} a^{-\Xi_n+\frac{1}{2}}\frac{1}{1-a^{-\frac{2\Xi_n+1}{2}(b-1)}}
       <\epsilon
    \end{align*}
   for some $a$ large enough, depending on $b$ and $\epsilon$.
   Similarly, $\|w-w^\flat\|_0\leq  
2\tau\sum_{q\geq0}
a^b\left(a^{\frac{1-b-2c}{2}}\right)^{b^q}<2\tau a^{-\Xi_n+1}
<\epsilon$ for $a$ large enough.
 Therefore,  Theorem \ref{main theorem} is proved provided that $a$ is chosen sufficiently large, depending on $b>1$, $c>\Xi_n+\frac{1}{2}$,  $C_*>1$, $K>1$, $\EC>1$, $\tau$ and $\epsilon>0$.

\bibliographystyle{abbrv}
\bibliography{ref}

\Addresses

\end{document}